\documentclass[12pt]{amsart}

\usepackage{amssymb,amsthm,amsmath}

\RequirePackage[dvipsnames,usenames]{color}
\input{kmacros3.sty}
\usepackage{mabliautoref}

\usepackage{tikz}
\usepackage{tikz-cd}
\usetikzlibrary{cd}
\usepackage{graphicx}
\usepackage[all,cmtip]{xy}

\usepackage{enumitem}
\numberwithin{equation}{theorem}

\usepackage{bm}
\usepackage{pifont}
\usepackage{upgreek}

\usepackage{eucal}
\usepackage{ulem}
\usepackage{stmaryrd}
\usepackage{xspace}

\usepackage{fullpage}

\usepackage{calc}

\usepackage{verbatim}
\usepackage{alltt}
\usepackage{scalerel}

\def\todo#1{\textcolor{red}%
{\footnotesize\newline{\color{red}\fbox{\parbox{\textwidth-15pt}{\textbf{todo: } #1}}}\newline}}

\renewcommand{\O}{\mathcal O}

\DeclareMathOperator{\KH}{KH}
\DeclareMathOperator{\Hir}{Hir}
\DeclareMathOperator{\BEH}{LHir}
\DeclareMathOperator{\BE}{L}

\begin{document}

\title{Thresholds of singularities in characteristic zero}
\author{Sandra Rodr\'iguez-Villalobos}
\address{Department of Mathematics, University of Utah, Salt Lake City, UT 84112, USA}
\email{rodriguez@math.utah.edu}
\author{Karl Schwede}
\address{Department of Mathematics, University of Utah, Salt Lake City, UT 84112, USA}
\email{schwede@math.utah.edu}
\begin{abstract}
    We study properties of characteristic zero variants of Frobenius thresholds $c^J(\fra)$ inspired by the work of Epstein-McDonald-R.G. and the second author.  Suppose $R$ is an excellent domain of equal characteristic zero with a dualizing complex and $\fra, J \subseteq R$ are nonzero ideals with $\fra \subseteq \sqrt{J}$.  By using a log resolution of singularities $Y$ of $(R, \fra)$ as well as the derived global sections of a line bundle on $Y$, we construct two analogs of the Frobenius threshold, the Hironaka-threshold and the Koszul-Hironaka threshold.  We show that these two thresholds agree for parameter ideals and are always rational numbers.  We prove that these thresholds also share many of the properties of the classical Frobenius threshold, especially for parameter ideals.  If $R$ is regular, we show that the set of thresholds coincides with the set of possible multiplier-ideal-jumping numbers.  Finally, we show that for parameter ideals in a Kawamata log terminal ring, our thresholds coincide with the limit of the Frobenius thresholds of the mod-$p$-reductions as $p$ goes to infinity.  There is even a version of this without that Kawamata log terminal hypothesis.
\end{abstract}
\maketitle

\section{Introduction}

Inspired by the notion of log canonical thresholds from characteristic zero, in a Noetherian ring of characteristic $p > 0$, Huneke-\mustata-Takagi-Watanabe (\cite{HunekeMustataTakagiWatanabeFThresholdsTightClosureIntClosureMultBounds} ) defined the notion of the \emph{Frobenius threshold of an ideal $\fra$} relative to another ideal $J$ (with $\fra \subseteq \sqrt{J}$).  In the case that $(R, \fram)$ is regular local and $J = \fram$, this coincides with the $F$-pure threshold, the supremum of $\lambda$ such that $(R, \fra^{\lambda})$ is $F$-pure, but the definition makes sense more generally (\cite{MustataTakagiWatanabeFThresholdsAndBernsteinSato,TakagiWatanabeFPureThresh}).  Indeed, supposing $\fra, J$ are nonzero ideals in an $F$-finite Noetherian domain with $\fra \subseteq \sqrt{J}$, set 
$\nu_{\fra}^J(p^e) = \max\{t \in \bN \;|\; \fra^t \nsubseteq J^{[p^e]} \}$. Then the $F$-threshold is defined to be 
\[
    c^J(\fra) := \lim_{e \to \infty} {\nu_{\fra}^J(p^e) \over p^e},
\]
a limit which we know exists thanks to \cite{DeStefaniNunezBetancourtPerezFThresholdsAndRelated}.  A related variant replaces $J^{[p^e]}$ with its tight closure $(J^{[p^e]})^*$ in the definition of $\nu_{\fra}^J(p^e)$, and obtains a limit we denote by $c_*^J(\fra)$ (\cite[Definition 8.5]{MengMukhopadhyay.hFunctionHKDensityFrobeniusPoincare}).  Under moderate hypotheses, that variant can also be interpreted as the infimum of $t \in \bR$ such that $(\fra R^+)_{>t} \subseteq J B$ where $B$ is a sufficiently large balanced big Cohen-Macaulay $R^+$-algebra, see \cite[Corollary 3.9]{RodriguezSchwede.BCMThresholdsNonPrincipal}.
We refer readers to \cite{TrivediWatanabeHilbertKunzDensityFunctions,Trivdei.FThresholdsForProjectiveCurves,BadillaCespedesNunezBetantcourtRodriguezVillalobos.FVolumes,GonzalezJaramillo-VelezNunezBetancourt.FThresholdsAndTestIdealsOfThomSebastiani,JeffriesNunezBetancourtQuinlanGallego.BSTheoryForSingularPositive, MengMukhopadhyay.hFunctionHKDensityFrobeniusPoincare} for other recent results and computations related to $F$-thresholds and generalizations thereof.

While $F$-thresholds are generalizations of $F$-pure thresholds, which were inspired by the log canonical threshold from characteristic zero, there seems to be no characteristic zero version of $F$-thresholds.  The point of this paper is to explore two potential definitions of a characteristic zero analog of the $F$-threshold (or more precisely, $c_*^J(\fra)$) which are defined using the tools of characteristic zero, namely log resolutions of singularities and regular alterations.  

Roughly, our idea is that instead of using Frobenius or a sufficiently large big Cohen-Macaulay $R^+$-algebra $B$ (as in \cite{RodriguezSchwede.BCMThresholdsNonPrincipal}), we should use the differential graded algebra $\myR\Gamma(Z, \cO_Z)$ where $Z \to \Spec R$ is a log resolution of singularities (or log alteration), since $\myR\Gamma(Z, \cO_Z)$ is a Cohen-Macaulay object in the derived category $D(R)$ and since $Z$ can be viewed as some philosophical replacement for $B$.  To solve the question of how to expand the ideal $J$ to this algebra we follow the ideas of \cite{EpsteinMcDonaldRGSchwede} but we expand $\fra^t$ in a different way already common in the theory of multiplier ideals and singularities of the minimal model program.

  Suppose $R$ is an excellent domain of equal characteristic zero with a dualizing complex and that $\fra, J$ are nonzero ideals with $\fra \subseteq \sqrt{J}$.  Set $Z \to \Spec R$ to be a log resolution of $(R, \fra)$ with $\fra \cO_Z = \cO_Z(-D)$.  First we define the \emph{derived Hironaka threshold} to be 
 \[
        c_{\Hir}^J(\fra) = \inf\{ t \geq 0 \;|\; R \to R/J \otimes_{R}^{\myL} \myR\Gamma(\cO_Z(\lfloor t D \rfloor)) \text{ is zero in the derived category }\}.
    \]        
    Next fix a set of generators $J = (x_1, \dots, x_n)$.  Set $\Kos_{\bullet}(\underline{x})$ to be the Koszul complex on those generators.  Then we define the \emph{derived Koszul-Hironaka threshold} to be
    \[
        c_{\KH}^{J}(\fra) = \inf\{ t \geq 0 \;|\; R \to \Kos_{\bullet}(\underline{x}) \otimes_{R}^{\myL} \myR\Gamma( \cO_Z(\lfloor t D \rfloor)) \text{ is zero in the derived category }\}.
    \]
    It turns out that this is independent of the choice of generators $\underline{x}$, see \autoref{lem.KHThresholdIndependentOfGenerators}.  We also note that $c_{\Hir}^J(\fra) = c_{\Hir}^J(\overline{\fra})$ and likewise $c_{\KH}^J(\fra) = c_{\KH}^J(\overline{\fra})$ where $\overline{\fra}$ denotes the integral closure, since $D$ does not change (\autoref{prop.InitialProperties}).  Furthermore, we see in \autoref{prop.KHClosureAgnostic} that $c_{\KH}^J(\fra) = c_{\KH}^{J^{\KH}}(\fra)$ where $J^{\KH}$ is the Koszul-Hironaka closure of \cite{EpsteinMcDonaldRGSchwede}.
    
    While these definitions look complicated, they are actually quite straightforward to compute using a computer when $R$ is Gorenstein and if one knows the multiplier ideal $\mJ(R, \fra^t)$ as $t$ varies.  See \autoref{rem.ComputingClosures}.  

    Because $\cO_Z(\lfloor t D \rfloor)$ only changes at certain rational numbers, we immediately see that 
    \[ 
        c_{\Hir}^J(\fra), c_{\KH}^{J}(\fra) \in \bQ.
    \]
    Furthermore, the infima above are actually minima.  See \autoref{lem.InfsAreMins}.

In \autoref{prop.TwoDefinitionsAgreeForParameterIdeals} we show that these two notions agree if $J$ is a parameter ideal.  If $R$ is {\bf regular}, then by \autoref{cor.HironakaThresholdViaMultiplierIdeal} we have that 
\[
    c_{\Hir}^J(\fra) = \min \{ t \geq 0 \;|\; \mJ(R, \fra^t) \subseteq J \}.
\]

Using the vanishing theorems associated to resolution of singularities (relative Kawamata-Viehweg vanishing), many standard results for  $F$-thresholds also hold in our characteristic zero setting.  First, we recover a version of \cite[Theorem 3.3]{HunekeMustataTakagiWatanabeFThresholdsTightClosureIntClosureMultBounds}.

\begin{theoremA*}[\autoref{thm.IntegralClosureMainResult}]
    Suppose $R$ is an excellent local domain of equal characteristic zero with a dualizing complex, $x_1, \dots, x_n$ are part of a system of parameters and define $J = (x_1, \dots, x_n) \subsetneq R$.  Given an ideal $I$ with $\sqrt{J} \supseteq I\supseteq J$, we have that $c_{\KH}^J(I) = n$ if and only if $\overline{I}=\overline{J}$.
\end{theoremA*}

The following tool, used to prove the result above, might be of independent interest.  In \autoref{lem.intersectionBEH} we show the following analog of an observation of Hochster (found in \cite[Section 3]{LipmanTeissierPseudoRational}).  In the context above, we show that for any integer $\lambda > 0$
\[
    \bigcap_{c_1,\ldots, c_n} (x_1^{c_1},\ldots, x_n^{c_n})^{\KH} = \ker \big( R \to H_0(\BE^{\lambda}(\underline{x}) \otimes^{\myL} \myR\Gamma(Z, \cO_Z)) \big)
    \]
where the intersection runs over positive integers $c_i$ such that $c_1 + \dots + c_n = \lambda + n - 1$.  
Here $\BE^{\lambda}(\underline{x})$ is the Buchsbaum-Eisenbud $L$-complex (or equivalently, an appropriate Eagon-Northcott complex).  If $S = \bZ[X_1, \dots, X_n]$, then $L^{\lambda}(\underline{x})$ is quasi-isomorphic to the base change $S/(X_1, \dots, X_n)^{\lambda} \otimes_S^{\myL} R$ (where $S \to R$ is induced by sending $X_i \mapsto x_i$).

We also recover the multiplicity bound of \cite[Theorem 5.6]{HunekeMustataTakagiWatanabeFThresholdsTightClosureIntClosureMultBounds}, see \autoref{prop.MultiplicityBound}.  However, since this is in the regular case and we can describe the threshold via the multiplier ideal, the proof is essentially the same as the original.

We obtain the following analog of \cite[Section 2]{MustataTakagiWatanabeFThresholdsAndBernsteinSato}, compare with \cite{BlickleMustataSmithDiscretenessAndRationalityOfFThresholds} and also \cite[Corollary 5.6]{RodriguezSchwede.BCMThresholdsNonPrincipal}.
\begin{theoremB*}[\autoref{cor.JumpingEqualThresholds}]
    Suppose $R$ is an excellent regular domain of equal characteristic zero  with a dualizing complex.  Suppose that $\fra \subseteq R$ is a proper nonzero ideal.  Then the set of jumping numbers for $\fra$ is the same as the set of derived Hironaka thresholds for $\fra$ as ideals $J$, such that $\fra \subseteq \sqrt{J}$, vary.
\end{theoremB*}

Perhaps our most interesting result is that $c_{\KH}^{J}(\fra)$ is the limit of the $c^{J_p}_*(\fra_p)$ under reduction modulo $p > 0$ as $p \to \infty$ in the special case that $J$ is a parameter ideal.  This generalizes the famous result that the log canonical threshold is the limit of the $F$-pure thresholds as $p \to \infty$ at least when $R$ is regular local and $J$ is the maximal ideal.  We state this in a special case here, leaving arbitrary finite type over a characteristic zero field case to the text below.  Note that similar results for homogeneous ideals in two-dimensional graded rings can be found in \cite{TrivediWatanabeHilbertKunzDensityFunctions,Trivdei.FThresholdsForProjectiveCurves} without assuming $J$ is a parameter ideal.
\begin{theoremC*}[\autoref{thm.ReductionModuloP}]
    Suppose $R_{\bZ}$ is a finite type $\bZ$-algebra and $\fra_{\bZ}, J_{\bZ} \subseteq R_{\bZ}$ are nonzero ideals.   We let $R, \fra, J$ be their base changes to $\bQ$.  We also assume that $R$ is geometrically integral and that $\fra \subseteq \sqrt{J}$ are proper ideals.  Further let $R_p, \fra_p, J_p$ be their base changes to $\bF_p$.  Suppose that $J$ is a parameter ideal,  then 
    \[
        \lim_{p \to \infty}  c^{J_p}_{*}(\fra_{p}) = c^J_{\KH}(\fra) = c^J_{\Hir}(\fra).
    \]
    In particular, if $R$ is weakly $F$-regular type (for instance if $R$ has rational Gorenstein singularities), then $\lim_{p \to \infty}  c^{J_p}(\fra_{p}) = c^J_{\KH}(\fra) = c^J_{\Hir}(\fra)$.  
\end{theoremC*}

If $J$ is not a parameter ideal, we still obtain the inequality $c_*^{J_p}(\fra_p) \leq c_{\KH}^J(\fra)$ for $p \gg 0$, see \autoref{prop.EasyComparisonReductionModuloP}.  The key to proving the above result is a characterization of $c_{\KH}^J(\fra)$ using Koszul-Hironaka closure.  Namely, in \autoref{prop.ConnectionWithKHForParameterIdeals} we show that $t \geq c^J_{\KH}(\fra)$ if and only if for every finite domain extension $R \subseteq S$ we have that $(\fra S)_{\geq t} \subseteq (JS)^{\KH}$ where $(-)_{\geq t}$ denotes the fractional integral closure.  This characterization is inspired by the earlier work of \cite{RodriguezSchwede.BCMThresholdsNonPrincipal}.  In fact, using Bertini theorems, \autoref{prop.ConnectionWithKHForParameterIdeals} also shows that if $h$ is the product of $m > t$ general elements of $\fra$, then $t \geq c_{\KH}^J(\fra)$ if and only if 
\[
    h^{t/m} \in (JS)^{\KH}
\]
where $S \supseteq R$ is any finite extension such that $h^{t/m} \in S$ makes sense.  This can be checked using a single finite extension and so is amenable to reduction modulo $p \gg 0$.  Then we note that $(JS)^{\KH}$ reduces modulo $p \gg 0$ to $(J_p S_p)^*$ by \cite[Theorem 7.12]{EpsteinMcDonaldRGSchwede}, using that $J$ is a parameter ideal.

\subsection*{AI Acknowledgements}

While the main results and ideas were the work of the human authors, OpenAI's ChatGPT 5.5 Pro and 5.6 Sol Pro helped the authors by suggesting substantial mathematical content in proofs of \autoref{cor.GeneralElementChoiceNew}, \autoref{lem.intersectionBEH} and \autoref{lem.WeirdMapInDerivedCatIsZero}.  Additionally, the authors ran the draft through OpenAI's ChatGPT 5.6 Sol Pro for the purpose of proofreading.  All writing in this document was done by the authors based on their understanding of the mathematics.

\subsection*{Acknowledgements}

Sandra Rodr\'iguez-Villalobos was supported by  NSF Grant \#2101800.  Karl Schwede was supported by NSF Grants \#2501903 and \#2101800  as well as Simons Travel Support for Mathematicians SFI-MPSTSM-00013051. 

\section{Preliminaries}

We begin by reminding the reader of the definition of parameter ideals in a non-local ring.

\begin{definition}
    \label{def.ParameterIdealGeneral}
    Suppose $R$ is Noetherian and locally equidimensional.  We say an ideal $(f_1, \dots, f_n) = J \subseteq R$ with chosen generators $f_1, \dots, f_n$, denoted $\underline{f}$, is a \emph{parameter ideal} if for each prime ideal $Q \supseteq J$ we have that $\underline{f}$ forms part of a system of parameters for $R_Q$.
\end{definition}

\subsection{Fractional integral closure powers}

Suppose $S$ is an integral domain (possibly non-Noetherian) with field of fractions $K = K(S)$, and $\frb \subseteq S$ is a nonzero ideal.  Suppose $t \geq 0$ is a rational number.  

We define the following two fractional integral closure powers of $\frb$:
\[
    \frb_{\geq t} =  \{x\in S \mid x^b\in \overline{\frb^a} \text{ for some } a \in \bZ_{\geq 0} \text{ and } b \in \bZ_{>0} \text{ with } a/b \geq t\} 
\]
and likewise 
\[
    \frb_{>t}=\{x\in S \mid x^b\in \overline{\frb^a} \text{ for some } a \in \bZ_{\geq 0} \text{ and } b \in \bZ_{>0} \text{ with } a/b>t\}.
\]
We record the following fact since we do not know a place in the literature where it is stated in this generality.
\begin{lemma}
    \label{lem.FractionalPowersByValuations}
    With notation as above, $x \in \frb_{\geq t}$ if and only if there exists a pair of integers $a \geq 0, b > 0$ with $a/b \geq t$ such that for every valuation $v$ of $K$ centered on $S$, there exists $y \in \frb$ (depending on $v$) such that $b v(x) \geq a v(y)$.

    Likewise $x \in \frb_{> t}$ if and only if there exists a pair of integers $a \geq 0, b > 0$ with $a/b > t$ such that for every valuation $v$ of $K$ centered on $S$, there exists $y \in \frb$ (depending on $v$) such that $b v(x) \geq a v(y)$.
\end{lemma}
\begin{proof} 
    Note that $x^b \in \overline{\frb^a}$ if and only if for every valuation $v$ of $K$ centered on $S$ with associated valuation domain $V$, we have that $x^b \in \frb^a V$ (see \cite[Theorem 6.8.3]{HunekeSwansonIntegralClosure}).  This means there exists $y_1 \in \frb^a$ such that $b v(x) \geq v(y_1)$.  Writing $y_1 = \sum_{i=1}^r c_i z_i$ with $c_i \in S$ and where $z_i = \prod_{j=1}^a z_{ij}$ with $z_{ij} \in \frb$, we have that $b v(x) \geq v(z_i)$ for some $i$.  But then $b v(x) \geq a v(z_{ij})$ for some $j$.  Set $y = z_{ij}$.
    
    Conversely, if for every $v$ as above, $b v(x) \geq a v(y)$ for some $y \in \frb$, then certainly $x^b \in \frb^a V$ and so $x^b \in \overline{\frb^a}$.  
    Hence we see that $x^b \in \overline{\frb^a}$ if and only if for each valuation $v$ as above, $b v(x) \geq a v(y)$ for some $y \in \frb$.  The lemma follows in the case that $t > 0$, as what we really need is $a/b > 0$.  If $t  = 0$, and $a = 0$, then the first statement follows immediately as $\frb^0 = S$.  In the second case, we always have $a/b > 0$ so nothing changes.
\end{proof}
For more information in the Noetherian setting, see \cite[Section 10.5]{HunekeSwansonIntegralClosure}.  Of particular importance to us will be the fact that if $S$ is Noetherian, then we may restrict the $v$ in \autoref{lem.FractionalPowersByValuations} to the Rees valuations of $\frb$.  See \cite[Proposition 10.5.2(7)]{HunekeSwansonIntegralClosure}.

\subsection{Frobenius thresholds and variants}

We begin by recalling the definition of the Frobenius threshold and two variants.  

\begin{definition}
    \label{def.CharPThresholds}
    Suppose $R$ is an $F$-finite Noetherian domain of characteristic $p>0$ and let $J,\fra$ be nonzero ideals with $\fra\subseteq \sqrt{J}$.

    \begin{description}
        \item[Classical $F$-thresholds \cite{MustataTakagiWatanabeFThresholdsAndBernsteinSato,HunekeMustataTakagiWatanabeFThresholdsTightClosureIntClosureMultBounds}] Assuming $J$ is proper. we define
        $$c^J(\fra)=\lim_{e\to\infty}\frac{\nu_\fra^J(p^e)}{p^e}$$
        where $\nu_\fra^J(p^e)=\max\{n\geq 0\mid \fra^n\not\subseteq J^{[p^e]}\}$.
        \item[{$F$-thresholds up to tight closure \cite[Definition 8.5]{MengMukhopadhyay.hFunctionHKDensityFrobeniusPoincare}}] We define
        $$c_*^J(\fra)=\lim_{e\to\infty}\frac{\nu_\fra^{J*}(p^e)}{p^e}$$
        where $\nu_\fra^{J*}(p^e)=\max\{n\geq 0\mid \fra^n\not\subseteq (J^{[p^e]})^*\}$.
        \item[$R^{+}$-thresholds] We define
        $$c_+^J(\fra)=\inf\{t\in \bQ_{\geq 0} \mid (\fra R^{+})_{>t}\subseteq JR^{+}\}.$$
    \end{description}
    Finally, if $J = R$, then we define all three thresholds to be zero.
\end{definition}

The first limit was shown to exist in full generality in \cite{DeStefaniNunezBetancourtPerezFThresholdsAndRelated}. The second limit was shown to exist in \cite[Lemma 8.6]{MengMukhopadhyay.hFunctionHKDensityFrobeniusPoincare}. 

We begin with an alternate characterization of $c_+^J(\fra)$.  

\begin{lemma}
    \label{lem.AlternateCharacterizationOfPlusThreshold}
    With notation as in \autoref{def.CharPThresholds}, we have that 
        \[
            c_+^J(\fra) = \lim_{e \to \infty} {\nu^{J+}_{\fra}(p^e) \over p^e}
        \]    
        where $\nu_{\fra}^{J+}(p^e) = \max\{ n \geq 0 \;|\; \fra^n \not\subseteq (J^{[p^e]})^+ \}$.
\end{lemma}
\begin{proof}
    As we are taking an infimum, it is straightforward to see that $c_+^J(\fra) = \inf \{ t = {a \over p^e} \;|\; a \in \bZ_{\geq 0}, (\fra R^+)_{\geq{a \over p^e}} \subseteq JR^+\}$.  As Frobenius is an isomorphism on $R^+$, this is equal to 
    \begin{align*}
        \inf \Big\{ t = {a \over p^e} \;\Big|\; a,e \in \bZ_{\geq 0}, \overline{\fra^{a} R^+} \subseteq {J^{[p^e]}}R^+\Big\}
        =& \inf_e \Bigg\{\frac{\min\{ a \in \bZ_{\geq 0}, \;\Big|\;  \overline{\fra^{a} R^+} \subseteq {J^{[p^e]}}R^+\}}{p^e}\Bigg\}.
    \end{align*}
    Note that, since Frobenius gives a bijection from $(\fra R^+)_{\geq a}$ to $(\fra R^+)_{\geq pa}$ and from $J^{[p^e]} R^+$ to $J^{[p^{e+1}]} R^+$, we have that if $(\fra R^+)_{\geq a} \subseteq J^{[p^e]}R^+$ then $(\fra R^+)_{\geq pa}\subseteq J^{[p^{e+1}]}R^+$. Thus, it follows that for every fixed $e \geq 0$,
    \[
        \begin{array}{rcl}
        \displaystyle{\Big\{ t = {a \over p^e} \;\Big|\; a \in \bZ_{\geq 0}, \overline{\fra^{a} R^+} \subseteq {J^{[p^e]}}R^+\Big\}} &
        \subseteq & 
        \displaystyle{\Big\{ t = {a \over p^{e+1}} \;\Big|\; a \in \bZ_{\geq 0}, \overline{\fra^{a} R^+} \subseteq {J^{[p^{e+1}]}}R^+\Big\}}.
        \end{array}
    \]
    As a consequence, the sequence $\displaystyle{\Big\{\frac{\min\{ a \in \bZ_{\geq 0} \;|\;  \overline{\fra^{a} R^+} \subseteq {J^{[p^e]}}R^+\}}{p^e}\Big\}_e}$ is nonincreasing. Hence, 
    \begin{align*}
        \inf_e \Bigg\{\frac{\min\big\{ a \in \bZ_{\geq 0} \;\Big|\;  \overline{\fra^{a} R^+} \subseteq {J^{[p^e]}}R^+\big\}}{p^e}\Bigg\}
        &= \lim_{e\to\infty}\frac{\min\big\{ a \in \bZ_{\geq 0} \;\Big|\;  \overline{\fra^{a} R^+} \subseteq {J^{[p^e]}}R^+\big\}}{p^e}
    \end{align*}
        
    Note that $\myR\Gamma(Y, \cO_Y) \cong R^+$ for all $Y \to \Spec R^+$ blowups of finitely generated ideals (\cite{BhattDerivedDirectSummand}).  
    Therefore, as $\fra$ is finitely generated, there exists an integer $k > 0$ such that $\overline{\fra^{a} R^+} \subseteq \fra^{a-k} R^+$ for all $a \geq k$ by \cite{MaMcDonaldRGSchwede.BrianconSkoda}.  So, for each $e$, we have
    \[
        \begin{array}{rcl}
        \displaystyle{\frac{\min \Big\{ a \in \bZ_{\geq 0}\;\Big|\; \fra^{a} R^+ \subseteq  {J^{[p^e]}}R^+\Big\}+k}{p^e} }
        &=& \displaystyle\frac{\min\Big\{ a \in \bZ_{\geq k}\;\Big|\; \fra^{a-k} R^+ \subseteq  {J^{[p^e]}}R^+\Big\}}{p^e} \\[15pt]
        &\geq &  \displaystyle\frac{\min\Big\{ a \in \bZ_{\geq 0}\;\Big|\;  \overline{\fra^{a} R^+} \subseteq {J^{[p^e]}}R^+\Big\}}{p^e} \\[15pt]
        &\geq & \displaystyle\frac{\min\Big\{ a \in \bZ_{\geq 0}\;\Big|\;  \fra^{a} R^+\subseteq {J^{[p^e]}}R^+\Big\}}{p^e}.
        \end{array}
    \]
    It follows that 
    \[
        \begin{array}{rcl}
            \displaystyle\lim_{e\to\infty}\displaystyle\frac{\min\Big\{ a \in \bZ_{\geq 0} \;\Big|\;  \overline{\fra^{a} R^+} \subseteq {J^{[p^e]}}R^+\Big\}}{p^e}&= &
            \displaystyle\lim_{e\to\infty}\displaystyle\frac{\min\Big\{ a \in \bZ_{\geq 0} \;\Big|\;  \fra^{a} R^+ \subseteq {J^{[p^e]}}R^+\Big\}}{p^e}\\[15pt]
            &=&\displaystyle\lim_{e\to\infty}\displaystyle\frac{\max\Big\{ a \in \bZ_{\geq 0} \;\Big|\;  \fra^{a} R^+ \not\subseteq {J^{[p^e]}}R^+\Big\}+1}{p^e}\\[15pt]
            &=&\displaystyle\lim_{e\to\infty}\displaystyle\frac{\max\Big\{ a \in \bZ_{\geq 0} \;\Big|\;  \fra^{a} R^+ \not\subseteq {J^{[p^e]}}R^+\Big\}}{p^e}.
        \end{array}
   \]
   Combining this with our earlier observations completes the proof.
   \end{proof}

We recall the following result which is essentially implicit in \cite{RodriguezSchwede.BCMThresholdsNonPrincipal}.

\begin{lemma}
    \label{lem.ComparisonOfCharpThresholds}
    With notation as in \autoref{def.CharPThresholds}, we have $c^J_*(\fra) \leq c^J_+(\fra) \leq c^J(\fra)$.    
    If $R$ is weakly $F$-regular, then all three coincide. 
\end{lemma}
\begin{proof}
    This follows from \autoref{lem.AlternateCharacterizationOfPlusThreshold} since $J^{[p^e]} \subseteq (J^{[p^e]})^+ \subseteq (J^{[p^e]})^*$ with equality when $R$ is weakly $F$-regular.
\end{proof}

Likewise, we also have the following.

\begin{proposition}
    With notation as in \autoref{def.CharPThresholds}, if $J$ is a parameter ideal, then $c^J_+(\fra) = c^J_*(\fra)$.
\end{proposition}
\begin{proof}
    This follows from \autoref{lem.AlternateCharacterizationOfPlusThreshold} since $J^+ = J^*$ for parameter ideals by \cite{SmithTightClosureParameter}.
\end{proof}

\subsection{Multiplier modules and ideals and related vanishing theorems}
\label{subsec.MultiplierModulesAndVanishing}

    First we fix some notation.  If $Z \to \Spec R$ is a proper map and $\sF$ is a sheaf on $Z$, we write $\Gamma(\sF)$ instead of $\Gamma(Z, \sF)$ to help keep the notation more compact and since no confusion is likely.  Likewise, if $\sF^{\mydot} \in D(Z)$, then we write $\myR\Gamma(\sF^{\mydot})$ instead of $\myR\Gamma(Z, \sF^{\mydot})$.

    We now recall the notions of log resolution as well as an alteration variant.

    \begin{notation}
        Suppose that $X$ is an excellent Noetherian integral scheme of equal characteristic zero and $\fra \subseteq \cO_X$ is a nonzero ideal sheaf.  We say that a proper surjective map $\pi : Y \to X$ is a \emph{log alteration of $(X, \fra)$} if 
        \begin{enumerate}
            \item $\pi$ is generically finite,
            \item $Y$ is a nonsingular integral\footnote{warning, this assumption is not standard} scheme, and
            \item $\fra \cO_Y = \cO_Y(-D)$ for some divisor $D$ such that $D$ has simple normal-crossing support
        \end{enumerate}
        We say that $\pi$ is a \emph{log resolution} if additionally $\pi$ is birational and the exceptional set of $\pi$ is a divisor that has simple normal-crossing support with $D$.      We recall that log resolutions exist in this generality by \cite{TemkinDesingularizationOfQuasiExcellentCharZero}, also see \cite{HironakaResolution}.  

        Suppose $X = \Spec R$ and abuse notation to let $\fra$ also denote the corresponding ideal of $R$.  A log alteration or log resolution of $(R, \fra)$ is then a log resolution or log alteration of the corresponding pair $(\Spec R, \fra)$.  
    \end{notation}

    Now we review the definitions of multiplier ideals and multiplier modules.

    \begin{definition}
        \label{def.MultiplierModuleAndIdeal}
        Suppose $R$ is an excellent domain of equal characteristic zero with a dualizing complex and $\fra \subseteq R$ is a nonzero ideal.  Suppose that $\pi : Z \to \Spec R$ is a log resolution of singularities of the pair $(R, \fra)$ writing $\fra \cO_Z  = \cO_Z(-D)$.  For any $t \geq 0$, the \emph{multiplier module\footnote{also sometimes called the Grauert-Riemenschneider sheaf} of the pair $(R, \fra^t)$}, denoted $\mJ(\omega_R, \fra^t)$, is defined to be 
        \[
            \Gamma(\omega_Z \otimes \cO_Z(-\lfloor t D \rfloor)) = \Gamma(\cO_Z(\lceil K_Z - t D \rceil)) \subseteq \omega_R.
        \]
        Here the inclusion is obtained by applying Grothendieck duality to the composition $R \to \myR\Gamma(\cO_Z) \to \myR\Gamma(\cO_Z(\lfloor t D \rfloor))$ and taking the lowest nonzero cohomology.  The map is nonzero because $\pi$ is birational.  The map is an inclusion since the modules are torsion-free and rank one.  

        If additionally $R$ is $\bQ$-Gorenstein\footnote{meaning that $R$ is G1 (Gorenstein in codimension 1) and S2 and that $\omega_R^{(n)} =: R(nK_R)$ is locally free for some integer $n > 0$}, the \emph{multiplier ideal of the pair $(R, \fra^t)$} is 
        \[
            \Gamma(\cO_Z(\lceil K_Z - \pi^* K_R - t D \rceil)) \subseteq R
        \]
        where the inclusion is induced as follows. If we pick $R \subseteq \omega_R$ so that $K_R \geq 0$, then the image of $\Gamma(\cO_Z(\lceil K_Z - \pi^* K_R - t D \rceil)) \to \omega_R$ lands in $R$ since $R$ is S2 and a direct computation shows that the image lands in $R$ in codimension 1. 

        Clearly if $\omega_R \cong R$ (such as happens sufficiently locally if $R$ is quasi-Gorenstein), then we may identify the multiplier module with the multiplier ideal.  
        
        If $\fra = R$ or $t = 0$ then we sometimes omit $\fra^t$ from the notation and simply write $\mJ(\omega_R)$ and $\mJ(R)$.  If $\fra = (f)$ is principal, then we also write $\mJ(\omega_R, f^t)$ and $\mJ(R, f^t)$ for the corresponding multiplier module or ideal.
    \end{definition}

    The multiplier module and multiplier ideal are independent of the choice of log resolution.  
    More subtly, in the context above, we recall that 
    \begin{equation}
        \label{eq.RelativeKVVanishing}
        H^i(Z,\omega_Z \otimes \cO_Z(-\lfloor t D \rfloor)) = 0\;\;\;\; \text{ and } \;\;\;\;  H^i(Z, \cO_Z(\lceil K_Z - \pi^* K_R - t D \rceil)) = 0 
    \end{equation}
    for $i > 0$ by the generalization of relative Kawamata-Viehweg vanishing found in 
    \cite{Murayama.RelativeVanishingForQSchemes}.  
    
We assemble the above into the following.

\begin{lemma}
    \label{lem.DualOfMultiplierIdeal}
    Using the notation of \autoref{def.MultiplierModuleAndIdeal}, assume additionally that $R$ is Gorenstein, fix $\omega_R \cong R$ and so set $K_R = 0$.   Then the map  
    \[
        \mJ(R, \fra^t) \cong \myR \Hom_{R} (\myR\Gamma(\cO_Z(\lfloor tD \rfloor)), R) \to \myR \Hom_R(R, R) \cong R
    \]
    is the inclusion of the multiplier ideal into $R$.
    
    Hence also, the map 
    \[
        \Hom_{D(R)} (\myR\Gamma(\cO_Z(\lfloor tD \rfloor)), R) \to \Hom_{D(R)}(R, R) \cong R
    \]
    is injective and the image is the multiplier ideal $\mJ(R, \fra^t)$.  Explicitly, $\phi : \myR\Gamma(\cO_Z(\lfloor tD \rfloor)) \to R$ is sent to the element $r \in \mJ(R, \fra^t) \subseteq R$ such that $R \to \myR\Gamma(\cO_Z(\lfloor tD \rfloor)) \xrightarrow{\phi} R$ is multiplication by $r$.
\end{lemma}
\begin{proof}
    As $K_R = 0$, the first statement follows quickly from the definitions  thanks to \autoref{eq.RelativeKVVanishing}.
    For the second statement, we recall that $\myH^0 \myR\Hom_R(A^{\mydot}, B^{\mydot})$ is naturally isomorphic to $\Hom_{D(R)}(A^{\mydot}, B^{\mydot})$, see \cite[\href{https://stacks.math.columbia.edu/tag/0A64}{Tag 0A64}]{stacks-project}.  Unraveling the definition proves the final statement.
\end{proof}

\begin{remark}
    \label{rem.AlterationDescriptionOfMultiplier}
    The argument of \cite[Theorem 8.1]{BlickleSchwedeTuckerTestAlterations} also implies that for any log alteration $Y \to \Spec R$ of $(R, \fra)$ with $\fra \cO_Y = \cO_Y(-D)$, we have that 
    \begin{equation}
        \label{eq.MultiplierModuleDescription}
        \mJ(\omega_R, \fra^t) = \Image\Big( \Gamma(Y, \omega_Y \otimes \cO_Y(-\lfloor t D \rfloor)) \xrightarrow{\Tr} \omega_R \Big)
    \end{equation}
    where $\Tr : K(Y) \to K(\Spec R)$ is the field trace.  The analogous result also holds for the multiplier ideal.
\end{remark}

\subsection{Koszul-Hironaka and Hironaka closure operations}

\begin{definition}
Suppose $R$ is a Noetherian excellent domain of equal characteristic zero.  We additionally suppose $R$ has a dualizing complex.  Fix $Z \to \Spec R$ a resolution of singularities.  For any ideal $I = (f_1, \dots, f_n) \subseteq R$ we define the \emph{Koszul-Hironaka (KH) closure} of $I$ to be 
\[
    I^{\KH} := \Ker\big(R \to H_0(\Kos_{\bullet}(\underline{f}) \otimes^{\myL} \myR\Gamma(\cO_Z)) \big)
\]  
This is independent of the choice of generators $\underline{f}$ and of the resolution $Z$ (\cite[Proposition 3.3]{EpsteinMcDonaldRGSchwede}).  

Similarly, we define the \emph{Hironaka pre-closure} of $I$ to be 
\[
    I^{\Hir} := \Ker\big(R \to H_0(R/I \otimes^{\myL} \myR\Gamma(Z, \cO_Z)) \big).
\]  
\end{definition}

We record some useful properties of these operations for use later.

\begin{lemma}
    \label{lem.KHHirBasicProperties}
    We fix $R$ to be an excellent domain of equal characteristic zero with a dualizing complex and suppose that $I, J \subseteq R$ are ideals.
    \begin{enumerate}
        \item $I^{\KH}$ is independent of the choice of generators $f_1, \dots, f_n$.  (\cite[Proposition 3.3]{EpsteinMcDonaldRGSchwede})    \label{lem.KHHirBasicProperties.IndependenceOfChoices}
        \item Koszul-Hironaka closure is idempotent\footnote{technically this is part of being a closure operation, but we emphasize it as it will be used later}, that is $(I^{\KH})^{\KH} = I^{\KH}$. (\cite[Proposition 3.3]{EpsteinMcDonaldRGSchwede})  \label{lem.KHHirBasicProperties.Idempotence}
        \item $I \subseteq I^{\KH} \subseteq I^{\Hir} \subseteq \overline{I}$ (\cite[Proposition 3.3, Equation (6.1.1) and Remark 6.6]{EpsteinMcDonaldRGSchwede})    \label{lem.KHHirBasicProperties.Containments}
        \item If $I \subseteq J$, then $I^{\KH} \subseteq J^{\KH}$ and $I^{\Hir} \subseteq J^{\Hir}$.  (\cite[Proposition 3.3]{EpsteinMcDonaldRGSchwede}, the argument for ${-}^{\Hir}$ is the same)    \label{lem.KHHirBasicProperties.OrderPreserving}
        \item The formation of $I^{\KH}$ and $I^{\Hir}$ commutes with localization.  (\cite[Proposition 3.5]{EpsteinMcDonaldRGSchwede}, the argument for ${-}^{\Hir}$ is the same)    \label{lem.KHHirBasicProperties.Localization}
        \item If $R \subseteq S$ is a finite extension of domains then $(IS)^{\KH} \cap R = I^{\KH}$.  (\cite[Proposition 3.7]{EpsteinMcDonaldRGSchwede})     \label{lem.KHHirBasicProperties.FiniteExtensions}
    \end{enumerate}
\end{lemma}

\section{Splitting in the derived category}

    When we study characteristic zero, we will need to understand splitting conditions in the derived category.  We state some facts that we believe are well known to experts.

    \begin{lemma}
        \label{lem.SplittingForAlterationsNormalCrossingSupport}
        Suppose $Z$ is a non-singular Noetherian excellent integral scheme with a dualizing complex over a field of characteristic zero.  Suppose further that $0 \leq \Delta$ is a $\bQ$-divisor with $\lfloor \Delta \rfloor = 0$ and such that $\Delta$ has simple normal-crossings support.  Assume that $\pi : Y \to Z$ is a normal integral alteration.  Then 
        \[
            \cO_Z \to \myR \pi_* \cO_Y(\lfloor \pi^* \Delta\rfloor)
        \]
        splits in the derived category.  
    \end{lemma}
    \begin{proof}
        We begin with a claim.
        \begin{claim}
            For any alteration $\mu : Y' \to Y$ with $Y'$ normal integral, we have an induced map $\cO_Y(\lfloor \pi^* \Delta\rfloor) \to \myR \mu_* \cO_{Y'}(\lfloor \mu^* \pi^* \Delta \rfloor)$.
        \end{claim}
        \begin{proof}[Proof of claim]
            It suffices to show $\cO_Y(\lfloor \pi^* \Delta\rfloor) \subseteq \mu_* \cO_{Y'}(\lfloor \mu^* \pi^* \Delta \rfloor)$.  Working on some open set, we see that $\Div_Y(f) + \lfloor \pi^* \Delta\rfloor \geq 0$ if and only if $\Div_Y(f) + \pi^* \Delta \geq 0$, but that is equivalent to $\Div_{Y'}(f) + \mu^* \pi^* \Delta \geq 0$ which implies $\Div_{Y'}(f) + \lfloor \mu^* \pi^* \Delta\rfloor \geq 0$.
        \end{proof}
        Hence, because of the claim, we may replace $Y$ with a larger alteration $Y'$ since splitting from $Y'$ will induce the splitting from $Y$.  Hence we may assume that $Y$ is regular.  Likewise, we may assume that $\pi^* \Delta$ is a Cartier divisor with integer coefficients.

        Applying $\myR \sHom_{Z}(-, \cO_Z)$ (and noting that $K_Z$ is a line bundle) we obtain a map:
        \[
            \cO_Z \xleftarrow{\Tr} \myR \pi_* \cO_Y(K_Y - \pi^* K_Z - \pi^* \Delta).
        \]
        It suffices to show that this map is split surjective as applying $\myR \sHom_{Z}(-, \cO_Z)$ again recovers the original map.  There are no higher direct images by relative Kawamata-Viehweg vanishing, proved in this generality in \cite{Murayama.RelativeVanishingForQSchemes}.  Hence it suffices to show that 
        \[
            \cO_Z \xleftarrow{\Tr} \pi_* \cO_Y(K_Y - \pi^* K_Z - \pi^* \Delta).
        \]
        is split surjective.
        Now, by \cite[Theorem 9.5.42]{LazarsfeldPositivity2} and \cite[Theorem 8.1]{BlickleSchwedeTuckerTestAlterations}, we know that 
        \[
            \mJ(\cO_Z, \Delta) = \cO_Z \cap \pi_* \cO_Y(K_Y - \pi^* K_Z - \pi^* \Delta) = \Tr\Big(\pi_* \cO_Y(K_Y - \pi^* K_Z - \pi^* \Delta)\Big).
        \]
        Since $(Z, \Delta)$ is KLT, $\mJ(\cO_Z, \Delta) = \cO_Z$ and hence from the displayed equation above, we see that $1 \in \pi_* \cO_Y(K_Y - \pi^* K_Z - \pi^* \Delta)$.  Note $\Tr$ sends $1$ to a unit (as we are in characteristic zero) and consider the map  
        \[
            \cO_Z \xleftarrow{\Tr} \pi_* \cO_Y(K_Y - \pi^* K_Z - \pi^* \Delta) = \myR \pi_* \cO_Y(K_Y - \pi^* K_Z - \pi^* \Delta).
        \]
        Thus $\cO_Z \subseteq \pi_* \cO_Y(K_Y - \pi^* K_Z - \pi^* \Delta) \xrightarrow{\Tr} \cO_Z$ is an isomorphism as it is multiplication by that same unit.  Hence our map $\Tr$ is split surjective, which is what we wanted to prove.
    \end{proof}

    A variant of this will be quite useful for us.

    \begin{proposition}
        \label{prop.SplittingForAlterationsNormalCrossingSupport}
        Suppose $Z$ is a smooth excellent integral scheme with a dualizing complex in characteristic zero.  Suppose $H$ is a $\bQ$-divisor on $Z$ such that the fractional part $\{ H \}$ of $H$ has simple normal-crossings support.  Then 
        \[
            \cO_Z(\lfloor H \rfloor ) \to \myR \pi_* \cO_Y(\lfloor \pi^* H \rfloor)
        \]
        splits in the derived category for any normal integral alteration $\pi : Y \to Z$.
    \end{proposition}
    \begin{proof}
        By passing to a a further alteration of $Y$, We may assume that $Y$ is regular and that $\pi^* H$ has integer coefficients and hence is a Cartier divisor.  
        By the projection formula, it suffices to show that $\cO_Z \to \myR \pi_* \cO_Y(\pi^* H - \pi^* \lfloor H \rfloor)$ splits.  But 
        \[
            \pi^* H - \pi^* \lfloor H \rfloor = \pi^* (H - \lfloor H \rfloor) = \pi^* (\{ H \}).
        \]  
        The result follows from \autoref{lem.SplittingForAlterationsNormalCrossingSupport}.
    \end{proof}


\section{The derived threshold in characteristic zero}

We define the invariants of interest in this paper and provide some alternate characterizations of them.

Recall the following notation from \autoref{subsec.MultiplierModulesAndVanishing}.
For a proper map $\pi : Y \to \Spec R$, and any sheaf $\sF$ on $Y$, we write 
\[
    \myR\Gamma(\sF) := \myR\Gamma(Y, \sF).
\]

\begin{definition}
    \label{def.NonPrincipalDerivedThreshold}
    Suppose that $R$ is an excellent domain of equal characteristic zero with a dualizing complex, and set $X = \Spec R$.  Suppose $\fra, J \subseteq R$ are nonzero ideals with $\fra \subseteq \sqrt{J}$.  Set $V = V(J) \subseteq X$.  Let $\pi : Z \to X$ be a log resolution of $(X, \fra)$ and write $\fra \cdot \cO_Z = \cO_Z(-D)$.  Then we define 
    \[
        c_{\Hir}^J(\fra) = \inf\{ t \geq 0 \;|\; R \to R/J \otimes_{R}^{\myL} \myR\Gamma(\cO_Z(\lfloor t D \rfloor)) \text{ is zero in the derived category }\}.
    \]    
    We call it the \emph{derived Hironaka threshold}.  In the case that $\fra = (f)$ is principal, we write $c_{\Hir}^J(f) = c_{\Hir}^J(\fra)$.

    We have the following variant.  With notation as above, fix a set of generators $J = (x_1, \dots, x_n)$.  Set $\Kos_{\bullet}(\underline{x})$ to be the Koszul complex on those generators.  Then we define
    \[
        c_{\KH}^{\underline{x}}(\fra) = \inf\{ t \geq 0 \;|\; R \to \Kos_{\bullet}(\underline{x}) \otimes_{R}^{\myL} \myR\Gamma( \cO_Z(\lfloor t D \rfloor)) \text{ is zero in the derived category }\}.
    \]
    We call it the \emph{derived Koszul-Hironaka threshold} and denote the principal case as $c_{\KH}^{\underline{x}}(f)$.  
\end{definition}
Note $\myR\Gamma(\cO_Z(\lfloor tD \rfloor))$ is independent of the choice of log resolution.  Indeed, its Grothendieck dual is $\myR\Gamma(\omega_Z \otimes \cO_Z(-\lfloor t D \rfloor))$ which is isomorphic to $\mJ(\omega_R, \fra^t)$ by \autoref{eq.RelativeKVVanishing} (\cite{Murayama.RelativeVanishingForQSchemes}) and that is independent of the choice of log resolution.  
It also turns out that $c_{\KH}^{\underline{x}}(\fra)$ depends only on $J = (\underline{x})$ and not on the choice of generators.

\begin{lemma}
    \label{lem.KHThresholdIndependentOfGenerators}
    With notation as in \autoref{def.NonPrincipalDerivedThreshold}, if $J = (x_1, \dots, x_n) = (y_1, \dots, y_m)$, then 
    \[
        c_{\KH}^{\underline{x}}(\fra) = c_{\KH}^{\underline{y}}(\fra)  
    \]
    and so we call the common value $c_{\KH}^J(\fra)$. 
\end{lemma}
\begin{proof}
    It suffices to consider the case $y_1 = x_1, \dots, y_n = x_n, y_{n+1} = g \in J$.  Consider the distinguished triangle 
    \[
        \Kos_{\bullet}(\underline{x}) \xrightarrow{g} \Kos_{\bullet}(\underline{x}) \to \Kos_{\bullet}(\underline{y}) \xrightarrow{+1}.  
    \] 
    The multiplication-by-$g$ map is nullhomotopic (\cf \cite[Remark 3.4]{EpsteinMcDonaldRGSchwede}) and hence zero in the derived category.  This means that there exists a map $\psi$ in the derived category such that  $\Kos_{\bullet}(\underline{x}) \to \Kos_{\bullet}(\underline{y}) \xrightarrow{\psi} \Kos_{\bullet}(\underline{x})$ is an isomorphism.  
    The result follows.
\end{proof}

\begin{remark}
    \label{rem.CaseJ=R}
    If $J = R$, then it is easy to see that $c^J_{\Hir}(\fra) = c^J_{\KH}(\fra) = 0$.
\end{remark}

Instead of using a log resolution of singularities, it is equivalent to check a log alteration.

\begin{lemma}
    \label{lem.WeCanUseLogAlterationInstead}    
    With notation as in \autoref{def.NonPrincipalDerivedThreshold} suppose $Y \to \Spec R$ is any alteration from a normal integral scheme that factors through a log alteration\footnote{for instance, a log resolution} $Y' \to \Spec R$ of $(R, \fra)$ with $\fra \cO_Y = \cO_Y(-G)$.  Then we have 
    \[
        c_{\Hir}^J(\fra) = \inf\{ t \geq 0 \;|\; R \to R/J \otimes_{R}^{\myL} \myR\Gamma(\cO_Y(\lfloor t G \rfloor)) \text{ is zero in the derived category }\}.
    \]
    and
    \[
        c_{\KH}^{J}(\fra) = \inf\{ t \geq 0 \;|\; R \to \Kos_{\bullet}(\underline{x}) \otimes_{R}^{\myL} \myR\Gamma( \cO_Y(\lfloor t G \rfloor)) \text{ is zero in the derived category }\}.
    \]
    In particular, for a fixed $t \geq 0$ and setting $B = R/J$ or $B = \Kos_{\bullet}(\underline{x})$, we have that  
    \[
        R \to B \otimes_{R}^{\myL} \myR\Gamma(\cO_Y(\lfloor t G \rfloor))
    \] 
    is zero in the derived category if and only if 
    \[ 
        R \to B \otimes_{R}^{\myL} \myR\Gamma(\cO_Z(\lfloor t D \rfloor))
    \]
    is zero in the derived category. 
\end{lemma}
\begin{proof}
    For the first statements, we may take infimum over $t \in \bQ_{\geq 0}$.  For the final statement, by slightly increasing $t$ without changing $\lfloor tG \rfloor$ or $\lfloor tD \rfloor$, we may assume that $t \in \bQ$.  Hence, we may restrict to rational $t$.  

    Suppose $Y_1 \to \Spec R$ is a log alteration of $(R, \fra)$ and $Y_2 \to Y_1 \to \Spec R$ is a further alteration with $Y_2$ normal and integral.  We write $\fra \cO_{Y_i} = \cO_{Y_i}(-D_i)$.  We shall first show that $R \to B \otimes_{R}^{\myL} \myR\Gamma(\cO_{Y_1}(\lfloor t D_1 \rfloor))$ is zero if and only if $R \to B \otimes_{R}^{\myL} \myR\Gamma(\cO_{Y_2}(\lfloor t D_2 \rfloor))$ is zero.  

    Certainly if $R \to R/J \otimes^{\myL}_R \myR\Gamma(\cO_{Y_1}(\lfloor t D_1 \rfloor))$ is zero in the derived category, so is $R \to R/J \otimes^{\myL}_R \myR\Gamma(\cO_{Y_2}(\lfloor t D_2 \rfloor))$, and likewise replacing $R/J$ with $\Kos_{\mydot}(\underline{x})$.  
    On the other hand notice that $\cO_{Y_1}(\lfloor t D_1 \rfloor) \to \myR \nu_* \cO_{Y_2}(\lfloor t D_2 \rfloor)$ splits by \autoref{prop.SplittingForAlterationsNormalCrossingSupport}.  

    Now, consider the following sets.
    \[
        S_{Y_i} = \{ t \geq 0 \;|\; R \to B \otimes_{R}^{\myL} \myR\Gamma(\cO_{Y_i}(\lfloor t D_i \rfloor)) \text{ is zero in the derived category }\}
    \]
    our work above shows that $S_{Y_1} = S_{Y_2}$ and hence their infima also agree.  

    If $\pi : Z \to \Spec R$ is a log resolution and $\nu : Y \to \Spec R$ is a log alteration, we can pick a further normal integral alteration $Y_2 \to \Spec R$ factoring through both $\pi$ and $\nu$ (and note we have $Y_2 \to Y \to Y' \to \Spec R$ where $Y'$ is a log alteration).  Then our observation above shows that $S_{Z} = S_{Y_2} = S_{Y'} = S_{Y}$ and so their infima also agree.  This completes the proof.  
\end{proof}


We prove some basic properties of $c^J_{\Hir}(\fra)$ and $c^J_{\KH}(\fra)$.  

\begin{proposition}
    \label{prop.InitialProperties}
    Suppose that $R$ is an excellent integral domain of equal characteristic zero with a dualizing complex.  Suppose further that $\fra, J \subseteq R$ are nonzero ideals with $\fra \subseteq \sqrt{J}$.    
    
    Then the following hold.
    \begin{enumerate}
        \item     \label{prop.InitialProperties.ComparisonOfTwoThresholds}
             $c^J_{\Hir}(\fra) \leq c^J_{\KH}(\fra)$.
        \item For any integer $l > 0$ we have that $c_{\Hir}^J(\fra) = l c_{\Hir}^J(\fra^l)$ and likewise $c_{\KH}^J(\fra) = l c_{\KH}^J(\fra^l)$. \label{prop.InitialProperties.PowersOfa}
        \item If $I \supseteq J$, then $c_{\Hir}^I(\fra) \leq c_{\Hir}^J(\fra)$ and likewise $c_{\KH}^I(\fra) \leq c_{\KH}^J(\fra)$.\label{prop.InitialProperties.BiggerJ}
        \item If $\frb \subseteq \fra$, then $c_{\Hir}^J(\frb) \leq c_{\Hir}^J(\fra)$ and likewise $c_{\KH}^J(\frb) \leq c_{\KH}^J(\fra)$.\label{prop.InitialProperties.Biggera}
        \item We have $c_{\Hir}^J(\fra) = c_{\Hir}^J(\overline{\fra})$ and $c_{\KH}^J(\fra) = c_{\KH}^J(\overline{\fra})$ where $\overline{(-)}$ denotes integral closure.\label{prop.InitialProperties.IntegralClosureAgnostic}
    \end{enumerate}
\end{proposition}
\begin{proof}
    For \autoref{prop.InitialProperties.ComparisonOfTwoThresholds}, note that we have a factorization $R \to \Kos(\underline{x}) \to R/J$.  Thus the set we take the infimum over to compute $c^J_{\Hir}(\fra)$ is at least as large as the one we use for $c^J_{\KH}(\fra)$.

    For \autoref{prop.InitialProperties.PowersOfa}, notice that $Z \to \Spec R$ is also a log resolution of $(R, \fra^l)$ and also that $\fra^l \cO_Z = \cO_Z(-l D)$.  The result quickly follows.

    For \autoref{prop.InitialProperties.BiggerJ}, if $I \supseteq J$, then we have a factorization $R \to R/J \to R/I$, and so for $I$ we are taking the infimum of a possibly larger set which yields a potentially smaller value, proving \autoref{prop.InitialProperties.BiggerJ} in the $c_{\Hir}$ case.  For $c_{\KH}$, choose explicit generators $\underline{x}$ for $J$ and then extend those to generators $\underline{y}$ for $I$.  This yields a map $\Kos_{\bullet}(\underline{x}) \to \Kos_{\bullet}(\underline{y})$, and then we proceed as above.

    Likewise if $\frb \subseteq \fra$, then if we take a common log resolution for both and write $\frb \cO_Z = \cO_Z(-E)$, we see that $D \leq E$.  Again, this yields a larger set to be taking the infimum over, yielding \autoref{prop.InitialProperties.Biggera}.  

    Finally, \autoref{prop.InitialProperties.IntegralClosureAgnostic} follows immediately from the definition as the divisor $D$ does not change.
\end{proof}

We make some other useful observations.  

\begin{lemma}
    \label{lem.InfsAreMins}
       With notation as in \autoref{def.NonPrincipalDerivedThreshold}, assuming that $c_{\Hir}^J(\fra)$ and $c_{\KH}^J(\fra)$ are finite (which we shall prove below in \autoref{lem.UpperBound}), the infima defining $c_{\Hir}^J(\fra)$ and $c_{\KH}^J(\fra)$ are actually minima.  In other words,
         \[
         \begin{array}{rclc}
            c_{\Hir}^J(\fra) & = & \min\{ t \geq 0 \;|\; R \to R/J \otimes_{R}^{\myL} \myR \Gamma( \cO_Z(\lfloor t D \rfloor) )\text{ is zero }\} & \text{ and }\\
            c_{\KH}^{J}(\fra) & = & \min\{ t \geq 0 \;|\; R \to \Kos_{\bullet}(\underline{x}) \otimes_{R}^{\myL} \myR \Gamma( \cO_Z(\lfloor t D \rfloor)) \text{ is zero }\}.
         \end{array}
        \]   
\end{lemma}
\begin{proof}
    Simply notice that $\myR\Gamma(\cO_Z(\lfloor t D \rfloor)) = \myR\Gamma( \cO_Z(\lfloor (t+\epsilon) D \rfloor))$ for $1 \gg \epsilon > 0$. 
\end{proof}

\begin{lemma}
    \label{lem.RationalThresholds}
    With notation as in \autoref{def.NonPrincipalDerivedThreshold}, $c^J_{\Hir}(\fra)$ and $c^J_{\KH}(\fra)$ are rational numbers.
\end{lemma}
\begin{proof}
        For $t \geq 0$, the divisor $\lfloor t D \rfloor$, and hence the complex $\myR\Gamma(\cO_Z(\lfloor t D \rfloor))$, only changes at a set of rational numbers $t \geq 0$ without accumulation points.    
The lemma follows.
\end{proof}

\begin{lemma}
    \label{lem.Localization}
    With notation as in \autoref{def.NonPrincipalDerivedThreshold}, we have that 
    \[
        \begin{array}{rclc}
            c_{\Hir}^J(\fra) & = & \max \{ c_{\Hir}^{J_{\fram}}(\fra_{\fram}) \;|\; \fram \in \mSpec R \} & \text{ and }\\
            c_{\KH}^J(\fra) & = & \max \{ c_{\KH}^{J_{\fram}}(\fra_{\fram}) \;|\; \fram \in \mSpec R \} 
        \end{array}
    \]
\end{lemma}
\begin{proof}
    Just as in \autoref{lem.RationalThresholds}, the only possible threshold values are those $t$ such that $\lfloor t D \rfloor \neq \lfloor (t - \epsilon) D \rfloor$ for all $\epsilon > 0$.  Furthermore, for $B = R/J$ or $\Kos_{\bullet}({\bf x}; R)$, where $J = ({\bf x})$, we have the map $R \to B \otimes^{\myL} \myR\Gamma(\cO_Z(\lfloor t D \rfloor))$ is zero if and only if it is zero after localizing at all maximal ideals.  The result follows.
\end{proof}

In \autoref{lem.UpperBound} below, we will show that both $c_{\Hir}^J(\fra)$ and $c_{\KH}^J(\fra)$ are finite, and if $J \neq R$ then positive, but, a priori, they could be infinite or zero.  To prove that, we use the following strategy which lets us reduce to the principal case in many situations.  This approach is also pervasive in the study of multiplier ideals, compare with \cite[Proposition 9.2.28]{LazarsfeldPositivity2}.  In our generality, it is essentially \cite[Theorem 10.1, Corollary 10.3]{LyuMurayama.RelativeMMPForExcellent} or \cite[Lemma 2.12]{McDonald.MutliplierIdealsViaDerivedSplittings}.

\begin{lemma}
    \label{prop.GeneralElementComparisonLocal}
    We adopt the notation of \autoref{def.NonPrincipalDerivedThreshold} with $\fra = (f_1, \dots, f_r)$.  Additionally assume that $(R, \fram, k)$ is local.  For any integer $m > 0$, choose in succession, $h_1, \dots, h_m \in \fra$ general $\bQ$-linear  
    combinations of the $f_i$  and set $h = \prod h_i \in \fra^m$.  
    Then for any $m > t \geq 0$, we have that $\Div_Z(h)$ has normal-crossing support, 
    \[ 
    \lfloor tD\rfloor =  \lfloor {t \over m} \Div_Z(h)\rfloor.
    \]
    Thus if $m > c^J_{\KH}(\fra)$ we see that  
    \[
    c^{J}_{\KH}(\fra) = m c^{J}_{\KH}(h). 
    \]
    Also if $m > c^J_{\Hir}(\fra)$ then we have that 
    \[ 
        c^{J}_{\Hir}(\fra) = m c^{J}_{\Hir}(h). 
    \]
\end{lemma}

\begin{proof}
    If $r = 1$, there is nothing to do so we may assume $r \geq 2$.
    It quickly follows from \cite[Theorem 10.1]{LyuMurayama.RelativeMMPForExcellent} (also see \cite[Corollary 10.3]{LyuMurayama.RelativeMMPForExcellent} and \cite[Lemma 2.12]{McDonald.MutliplierIdealsViaDerivedSplittings})
    that there exists an open dense set $W_1 \subseteq k^r$ such that for any $(\overline{a_1}, \dots, \overline{a_r}) \in W_1$, if we set $h_1 = a_1 f_1 + \dots + a_r f_r$, $Z$ is a log resolution of $(R, h_1)$ and in particular $\Div_Z h_1 = D + H_1$ has simple normal-crossing support where $H_1$ is nonsingular and has no common components with $D$.   As ${\bQ}^r \subseteq k^r$ is Zariski-dense, we may choose the coefficients $a_i$ from $\bQ$ and thus fix such a $h_1$ and corresponding $H_1$.
    Then there exists an open dense set $W_2 \subseteq k^r$ such that for any $(\overline{a_1}, \dots, \overline{a_r}) \in W_2$ if we set $h_2 = a_1 f_1 + \dots + a_r f_r$ then $Z$ is a log resolution of $(R, h_1h_2)$ and $\Div_Z (h_1 h_2) = 2D + H_1 + H_2$ has simple normal-crossing support and where $H_2$ is nonsingular  and has no common component with $D$ or $H_1$.  Arguing as above, we may again assume that the coefficients of $h_2$ are in $\bQ$.  Continuing in this way, we may construct $h_1, \dots, h_m$ and set $h = \prod_{i = 1}^m h_i$ so that $Z$ is a log resolution of $(R, h)$ and in particular $\Div_Z(h) = mD + H_1 + \dots + H_m$ is a simple normal crossing divisor with the $H_i$ reduced and with no common components with each other or $D$.  It follows that if $m > t \geq 0$ then $\lfloor tD\rfloor =  \lfloor {t \over m} \Div_Z(h)\rfloor$ as desired.
    Therefore $\myR \Gamma( \cO_Z(\lfloor tD\rfloor)) =  \myR \Gamma (\cO_Z(\lfloor {t \over m} \Div_Z(h)\rfloor))$.  The result follows.
\end{proof}

\begin{remark}
    If instead of being local, one assumes that $R$ is essentially finite type over a field $k$, then for successively chosen general $k$-linear (or $\bQ$-linear) combinations $h_1, \dots, h_m$ of the $f_i$, the same result holds by the argument of \cite[Proposition 9.2.28]{LazarsfeldPositivity2}.
\end{remark}

\begin{lemma}
    \label{lem.UpperBound}
    Suppose $J = (x_1, \dots, x_n)$ where $R$ is a Noetherian excellent domain of equal characteristic zero with a dualizing complex and $\fra \subseteq R$ is a nonzero ideal such that $\fra^l \subseteq J$ for some integer $l > 0$.  Pick $Z \to \Spec R$ a log resolution of $(R, \fra)$ and write $\fra \cO_Z = \cO_Z(-D)$.  Then we have that 
    \[
        R \to \Kos_{\bullet}(\underline{x}) \otimes^{\myL} \myR\Gamma(\cO_Z(\lfloor (nl) D \rfloor))
    \]
    is zero in the derived category.  As a consequence, $c^J_{\Hir}(\fra), c^J_{\KH}(\fra) \leq nl$ and in particular, both are finite numbers.  

    Finally, if $J$ is proper, both are positive.
\end{lemma}
\begin{proof}
    For the first statement, we may assume that $R$ is local as the displayed map is zero if and only if all its localizations are zero.    
    If $J = R$, then $c^J_{\Hir}(\fra) = c^J_{\KH}(\fra) = 0$, and hence there is nothing to show for the first statement.  Thus we may assume that $J$ is proper.
    By \autoref{prop.InitialProperties} \autoref{prop.InitialProperties.ComparisonOfTwoThresholds}, it suffices to prove the $\KH$ case.  
    Pick an integer $m > nl$ and choose $h$ as in \autoref{prop.GeneralElementComparisonLocal}.  Using that lemma, we must show that 
    \[
        R \to \Kos_{\bullet}(\underline{x})  \otimes^{\myL} \myR\Gamma(\cO_Z(\lfloor {nl \over m} \Div_Z(h) \rfloor))
    \]  
    is zero.  Pick a log alteration $Y$ of $(R, h)$, dominating $Z$, $Y \xrightarrow{\nu} Z \to \Spec R$ such that if $W = \Spec S = \Spec \Gamma(Y, \cO_Y)$, then ${nl \over m} \Div_W(h)$ is principal and equal to $\Div_W(g)$.  
    By \autoref{lem.WeCanUseLogAlterationInstead}, it suffices to show that 
    \[
        R \to \Kos_{\bullet}(\underline{x}; R)  \otimes^{\myL}_R \myR\Gamma(\cO_Y(\lfloor {nl \over m} \Div_Y(h) \rfloor))
    \]
    is zero.  Or equivalently, it suffices to show that $R \xrightarrow{\cdot g} \Kos_{\bullet}(\underline{x}; S) \otimes^{\myL}_S \myR\Gamma(\cO_Y)$ is zero.  
    By construction, $h \in \fra^{m}$ and so $g \in \overline{\fra^{nl} S} \subseteq \overline{J^n S}$ where $\overline{-}$ denotes integral closure.  But by \cite[Theorem 4.14]{EpsteinMcDonaldRGSchwede} we see that $\overline{J^n S} = \overline{(JS)^n} \subseteq (JS)^{\KH}$.  In other words, we see that $S \xrightarrow{\cdot g} \Kos(\underline{x}; S) \otimes_S^{\myL} \myR\Gamma(\cO_Y)$ is zero.  This implies the first statement.  The ``consequence'' follows immediately from the first statement. 

    To show positivity, by \autoref{lem.Localization} we may assume that $(R, \fram)$ is local and that $J \subseteq \fram$.  Suppose that $c^J_{\Hir}(\fra)  = 0$ and hence
    \[
        R \to R/J \otimes^{\myL} \myR\Gamma(\cO_Z(\lfloor tD \rfloor))
    \]
    is zero for sufficiently small $t > 0$.  But $\lfloor tD \rfloor = 0$ for such sufficiently small $t > 0$ and so $R \to R/J \otimes^{\myL} \myR\Gamma(\cO_Z)$ is zero.
    This implies that $1 \in J^{\Hir}$ which, since $J \subseteq \fram$,  contradicts the fact that Hironaka pre-closure is faithful (\cite[Proposition 6.2]{EpsteinMcDonaldRGSchwede}).  As $c^J_{\KH}(\fra) \geq c^J_{\Hir}(\fra)$, the positivity of $c^J_{\KH}(\fra)$ also follows.
\end{proof}

\begin{corollary}
    \label{cor.GeneralElementChoiceNew}
    Suppose that $R$ is a Noetherian excellent domain of equal characteristic zero with a dualizing complex and $\fra, J \subseteq R$ are nonzero ideals with $\fra \subseteq \sqrt{J}$ and where $\fra = (f_1, \dots, f_r)$.  Fix an integer $m > c^J_{\KH}(\fra)$.  We may (successively) pick general $\bQ$-linear combinations $h_1, \dots, h_m$ of the $f_i$ such that if $h = \prod_{i = 1}^m h_i$, then 
    \[
        c^J_{\KH}(\fra) = m c^J_{\KH}(h).
    \]
    Likewise if $m > c^J_{\Hir}(\fra)$, then we may pick $h$ as above such that 
    \[
        c^J_{\Hir}(\fra) = m c^J_{\Hir}(h).
    \]
\end{corollary}
\begin{proof}
    Pick $\frq \in \mSpec R$ such that $c^J_{\KH}(\fra) = c^{J_{\frq}}_{\KH}(\fra_{\frq})$ as in \autoref{lem.Localization}.  Then, in $R_{\frq}$, we may pick $h$ as described such that $c_{\KH}^{J_{\frq}}(\fra_{\frq}) = m c_{\KH}^{J_{\frq}}(h)$ thanks to \autoref{prop.GeneralElementComparisonLocal}.  Combining this, the fact that $h \in \fra^m$, and \autoref{prop.InitialProperties} \autoref{prop.InitialProperties.PowersOfa} \autoref{prop.InitialProperties.Biggera} we obtain 
    \[
        m c_{\KH}^{J}(h) \leq m c_{\KH}^J(\fra^m) = c_{\KH}^J(\fra) = c_{\KH}^{J_{\frq}}(\fra_{\frq}) = m c_{\KH}^{J_{\frq}}(h).
    \]
    But $m c_{\KH}^{J_{\frq}}(h) \leq m c_{\KH}^{J}(h)$ thanks to \autoref{lem.Localization} and so all the terms are equal proving the first case.  The case for $c^J_{\Hir}(\fra)$ is identical.
\end{proof}

\begin{corollary}
    \label{cor.AlternateWayOfPrincipalizing}
    Using the notation of \autoref{def.NonPrincipalDerivedThreshold}, fix $B$ to be either  $R/J$ or $\Kos_{\bullet}(\underline{x})$ where $J = (\underline{x})$.  
    Then, adopting the terminology of \autoref{cor.GeneralElementChoiceNew}, if $t \in \bQ_{\geq 0}$ is such that $m > t$, then  
    \[
        R \to B \otimes^{\myL} \myR\Gamma(\cO_Z(\lfloor tD \rfloor))
    \]
    is zero if and only if 
    \[
        R \to B \otimes^{\myL} \myR\Gamma(\cO_{Z'}(\lfloor {t \over m} \Div_{Z'} h \rfloor))
    \]
    is zero where $Z' \to \Spec R$ is a log resolution of $(R, h)$.
\end{corollary}
\begin{proof}
    This is an immediate corollary of \autoref{cor.GeneralElementChoiceNew} and the definitions of $c^J_{\Hir}(\fra)$ and $c^J_{\KH}(\fra)$.
\end{proof}


We note the following well-known lemma.

\begin{lemma}
    \label{lem.ZeroVSCohomologicalZero}
    Suppose $R$ is a ring and $M^{\mydot} \in D(R)$.  Then a map $f : R \to M^{\mydot}$ is zero in the derived category if and only if 
    \[
        H_0(f) : R \to H_0(M^{\mydot}) 
    \]
    is zero.  
\end{lemma}
\begin{proof}
    This follows from the identification $\Hom_{D(R)}(R[0], M^{\mydot}) \cong H_0(M^{\mydot})$ as $f$ is sent to $(H_0(f))(1) \in H_0(M^{\mydot})$.   
\end{proof}

As a corollary we obtain the following.

\begin{corollary}
    With notation as in \autoref{def.NonPrincipalDerivedThreshold}, we have that 
    \[
        c_{\Hir}^J(\fra) = \inf\{ t \geq 0 \;|\; R \to R/J \otimes_{R}^{\myL} \myR \Gamma( \cO_Z(\lfloor t D \rfloor)) \text{ is zero on cohomology }\}.
    \]
    Furthermore, we have that 
    \[
        c_{\KH}^{J}(\fra) = \inf\{ t \geq 0 \;|\; R \to \Kos_{\bullet}(\underline{x}) \otimes_{R}^{\myL} \myR \Gamma( \cO_Z(\lfloor t D \rfloor)) \text{ is zero on cohomology }\}.
    \]
\end{corollary}
\begin{proof}
    This follows immediately from \autoref{lem.ZeroVSCohomologicalZero}.
\end{proof}

\begin{remark}
    \label{rem.ComputingClosures}
    Building on the Macaulay2 (\cite{M2}) package included in the arXiv version of \cite{EpsteinMcDonaldRGSchwede}, one can compute $c_{\Hir}^J(\fra)$ and $c^J_{\KH}(\fra)$ in a Gorenstein ring as long as multiplier ideals $\mJ(R, \fra^t)$ are known as $t$ varies.  Note that one can obtain $\myR\Gamma(\cO_Z(\lfloor tD \rfloor))$ as the Grothendieck dual of $\mJ(R, \fra^t)$ and then one can use the functionality from the Complexes package \cite{ComplexesSource} to do the computation.
\end{remark}

\section{Finite extensions and connections with closure operations}

We first explore the behavior of these invariants under finite extensions.

\begin{proposition}
    \label{prop.BehaviorUnderFinite}
    Suppose that $R$ is a Noetherian excellent domain of equal characteristic zero with a dualizing complex.  Suppose further that $\fra, J = (x_1, \dots, x_n)$ are nonzero ideals with $\fra \subseteq \sqrt{J}$.  Suppose additionally that $R \subseteq R'$ is a finite inclusion of domains.  We write $J' = J R'$ and $\fra' = \fra R'$.
    
    Then $c^J_{\KH}(\fra) = c^{J'}_{\KH}(\fra')$ and $c^J_{\Hir}(\fra) \geq c^{J'}_{\Hir}(\fra')$.  
\end{proposition}
\begin{proof}
    Fix $\pi' : Z' \to X' = \Spec R'$ a log resolution of $(\Spec R', \fra')$, noting that it is also a log alteration of $(\Spec R, \fra)$.  Write $\fra' \cO_{Z'} = \fra \cO_{Z'} = \cO_{Z'}(-D')$.
    It follows from \autoref{lem.WeCanUseLogAlterationInstead} that  
    \[
        c^J_{\KH}(\fra) = \inf\{ t \geq 0 \;|\; R \to \Kos_{\bullet}(\underline{x}) \otimes_{R}^{\myL} \myR\Gamma( \cO_{Z'}(\lfloor tD' \rfloor)) \text{ is zero }\}.
    \]
    But, as the Koszul complex commutes with base change, we have
    \[
        \Kos_{\bullet}(\underline{x}) \otimes_{R}^{\myL} \myR\Gamma( \cO_{Z'}(\lfloor tD' \rfloor)) 
        \cong 
        \Kos_{\bullet}(\underline{x}; R') \otimes_{R'}^{\myL} \myR\Gamma( \cO_{Z'}(\lfloor tD' \rfloor)).
    \]
    Finally, 
    by \autoref{lem.ZeroVSCohomologicalZero} 
    we have that $R \to \Kos_{\bullet}(\underline{x}) \otimes_{R}^{\myL} \myR\Gamma( \cO_{Z'}(\lfloor tD' \rfloor)) $ is zero if and only if $1 \in R = H_0(R)$ is sent to $0 \in H_0\big(\Kos_{\bullet}(\underline{x}) \otimes_{R}^{\myL} \myR\Gamma( \cO_{Z'}(\lfloor tD' \rfloor)) \big)$.  But that is equivalent to asking that $R' \to \Kos_{\bullet}(\underline{x}; R') \otimes_{R'}^{\myL} \myR\Gamma( \cO_{Z'}(\lfloor tD' \rfloor))$ sends $1$ to zero on zeroth cohomology as well.  Applying \autoref{lem.ZeroVSCohomologicalZero}  again proves the statement for the derived Koszul-Hironaka threshold.
    
    The result for the derived Hironaka threshold is similar.  We simply use that we have a map
    \[
        R/J \otimes_{R}^{\myL} \myR\Gamma( \cO_{Z'}(\lfloor tD' \rfloor)) 
        \to 
        R'/J' \otimes_{R'}^{\myL} \myR\Gamma( \cO_{Z'}(\lfloor tD' \rfloor))
    \]
    instead of an isomorphism.  Thus if $R \to R/J \otimes_{R}^{\myL} \myR\Gamma( \cO_{Z'}(\lfloor tD' \rfloor))$ is zero, so is $R \to R'/J' \otimes_{R'}^{\myL} \myR\Gamma( \cO_{Z'}(\lfloor tD' \rfloor))$.  But, by \autoref{lem.ZeroVSCohomologicalZero} and arguing as above, this is equivalent to asking that $R' \to R'/J' \otimes_{R'}^{\myL} \myR\Gamma( \cO_{Z'}(\lfloor tD' \rfloor))$ is zero.  
    Hence we are taking an infimum of a potentially larger set to compute $c^{J'}_{\Hir}(\fra')$ and so the desired inequality follows.
\end{proof}

\subsection{Connections with characteristic zero closure operations}

Our perspective in this paper connects nicely with the equal characteristic zero closure operations introduced in \cite{EpsteinMcDonaldRGSchwede}.  Alternative approaches might also be possible via differential operators, see for instance \cite{BrennerJeffriesNunezBetancourtQuantifyingSingularities}, although we will not explore that perspective here.

\begin{proposition}
    \label{prop.ConnectionWithKHForParameterIdeals}
    Suppose $J = (x_1, \dots, x_n) \subseteq R$ where $R$ is a Noetherian excellent domain of equal characteristic zero with a dualizing complex and $\fra \subseteq \sqrt{J}$ is a nonzero ideal.  Fix $\pi : Z \to \Spec R$ a log resolution of $(R, \fra)$ with $\fra \cO_Z = \cO_Z(-D)$.  
    Choose an integer $m > c^J_{\KH}(\fra)$ and pick $h$ as in  \autoref{cor.GeneralElementChoiceNew}.  Fix $t \in \bQ_{\geq 0}$ with $m > t$.  Then the following are equivalent.  
        \begin{enumerate}
            \item $t \geq c^J_{\KH}(\fra)$.\label{prop.ConnectionWithKHForParameterIdeals.BiggerThanCJKH}
            \item For every finite domain extension $R \subseteq S$ we have that 
            \[
                (\fra S)_{\geq t} \subseteq (JS)^{\KH}
            \]
            where $(-)_{\geq t}$ denotes the fractional integral closure.
            \label{prop.ConnectionWithKHForParameterIdeals.KHOnS}
            \item For some finite domain extension $R \subseteq S_0$ and every further finite domain extension $S_0 \subseteq S$ we have that 
            \[
                (\fra S)_{\geq t} \subseteq (JS)^{\KH}
            \]
            where $(-)_{\geq t}$ denotes the fractional integral closure.
            \label{prop.ConnectionWithKHForParameterIdeals.KHOnSSuffBig}
            \item $h^{t/m} \in (J S)^{\KH}$ where $S \supseteq R$ is any (equivalently some) finite domain extension such that $h^{t/m}$ makes sense in $S$.  \label{prop.ConnectionWithKHForParameterIdeals.KHOnSGeneralElement}
        \end{enumerate}
        Furthermore,  if $x_1, \dots, x_n$ generate a parameter ideal (as defined in \autoref{def.ParameterIdealGeneral}), then the above is also equivalent to the following.
        \begin{enumerate}[resume]
            \item $(\fra S)_{\geq t} \subseteq (J S)^{\Hir}$ for every finite domain extension $S \supseteq R$.  \label{prop.ConnectionWithKHForParameterIdeals.HirOnS}

            \item $h^{t/m} \in (J S)^{\mathrm{Hir}}$ where $S \supseteq R$ is any (equivalently some) finite domain extension such that $h^{t/m}$ makes sense in $S$.  \label{prop.ConnectionWithKHForParameterIdeals.GeneralElementHirOnS}
        \end{enumerate}
\end{proposition}

In \autoref{prop.ConnectionWithKHForParameterIdeals.KHOnSGeneralElement}, note that there can be multiple versions of $h^{t/m} \in S$, but they all differ by units which does not impact the condition.  Thus it is harmless to suppress this issue.

\begin{proof}
    First, we justify the  ``(equivalently some)'' in \autoref{prop.ConnectionWithKHForParameterIdeals.KHOnSGeneralElement}.  For this, it suffices to observe that if $S \subseteq S'$ is a finite extension and $g \in S$, then $g \in (JS)^{\KH}$ if and only if $g \in (JS')^{\KH}$ by \cite[Proposition 3.7]{EpsteinMcDonaldRGSchwede}.

    Condition \autoref{prop.ConnectionWithKHForParameterIdeals.BiggerThanCJKH} is equivalent to requiring that $R \to \Kos_{\bullet}(\underline{x}) \otimes^{\myL} \myR\Gamma (\cO_Z(\lfloor tD \rfloor))$ is zero.  Pick a finite domain extension $S \supseteq R$ such that $h^{t/m} \in S$.  Fix $Y \to \Spec S$ a resolution of singularities such that $\nu : Y \to \Spec S \to \Spec R$ factors through $\pi : Z \to \Spec R$, and so that $Y$ is a log alteration of $(R, \fra)$ and $(R, h)$.   By \autoref{cor.AlternateWayOfPrincipalizing}, \autoref{prop.ConnectionWithKHForParameterIdeals.BiggerThanCJKH} is equivalent to requiring that $t \geq m c_{\KH}^J(h)$.   But by \autoref{lem.WeCanUseLogAlterationInstead}, that  
    is equivalent to $R \to \Kos_{\bullet}(\underline{x}; R) \otimes^{\myL}_R \myR\Gamma(\cO_Y({t \over m} \Div_Y(h)))$ being zero.
    
    As the Koszul complex commutes with base change, we see that \autoref{prop.ConnectionWithKHForParameterIdeals.BiggerThanCJKH} is equivalent to requiring that 
    \[
        R \to S \xrightarrow{\cdot h^{t/m}} S \to \Kos(\underline{x}; S) \otimes^{\myL}_S \myR \Gamma(\cO_Y) 
    \]
    is zero.  But, again using \autoref{lem.ZeroVSCohomologicalZero}, that is equivalent to requiring that $h^{t/m} \in (JS)^{\KH}$, which is \autoref{prop.ConnectionWithKHForParameterIdeals.KHOnSGeneralElement} and so we have proven that \autoref{prop.ConnectionWithKHForParameterIdeals.BiggerThanCJKH} $\Leftrightarrow$ \autoref{prop.ConnectionWithKHForParameterIdeals.KHOnSGeneralElement}. 

    Now consider \autoref{prop.ConnectionWithKHForParameterIdeals.KHOnS}.  Suppose first that $R \subseteq S$ is a finite domain extension.  Suppose $t \geq c^J_{\KH}(\fra) = c^{JS}_{\KH}(\fra S)$ where the equality is via \autoref{prop.BehaviorUnderFinite} and pick  $g \in (\fra S)_{\geq t}$.  For a log resolution $Y \to \Spec S$ with $\fra \cO_Y = \cO_Y(-D_Y)$, we have that $\Div_Y(g) \geq t D_Y$ by the definition of $(\fra S)_{\geq t}$.  
    Hence we have a factorization 
    \[
        \xymatrix@C=16pt{
            R \ar[r] & S \ar[r] &\Kos_{\bullet}(\underline{x}) \otimes^{\myL} \myR\Gamma(\cO_Y) \ar@/^2pc/[rr]^{\cdot g} \ar[r] & \Kos_{\bullet}(\underline{x}) \otimes^{\myL}\myR\Gamma(\cO_Y(\lfloor t D_Y \rfloor )) \ar[r] & \Kos_{\bullet}(\underline{x}) \otimes^{\myL} \myR\Gamma(\cO_Y).
        } 
    \]
    Our hypothesis implies that the composition of the first three maps is zero, hence the full composition is zero.  But this means that $g$ is in the kernel of $S \to H_0\big(\Kos_{\bullet}(\underline{x}; R) \otimes^{\myL}_R \myR\Gamma(\cO_Y) \big) \cong H_0\big(\Kos_{\bullet}(\underline{x}; S) \otimes^{\myL}_S \myR\Gamma(\cO_Y) \big)$ and so $g \in (JS)^{\KH}$ by definition.  Hence \autoref{prop.ConnectionWithKHForParameterIdeals.BiggerThanCJKH} $\Rightarrow$ \autoref{prop.ConnectionWithKHForParameterIdeals.KHOnS}.

    Now, certainly \autoref{prop.ConnectionWithKHForParameterIdeals.KHOnS} $\Rightarrow$ \autoref{prop.ConnectionWithKHForParameterIdeals.KHOnSSuffBig}.  Suppose \autoref{prop.ConnectionWithKHForParameterIdeals.KHOnSSuffBig} holds for some $S_0$.  We shall deduce \autoref{prop.ConnectionWithKHForParameterIdeals.KHOnS}.  It suffices to consider $R \subseteq T \subseteq S$ finite domain extensions where $S \supseteq S_0$ and hence by hypothesis, $(\fra S)_{\geq t} \subseteq (J S)^{\KH}$.  We must show the containment $(\fra T)_{\geq t} \subseteq (J T)^{\KH}$.  Now, $(\fra S)_{\geq t} \cap T = (\fra T)_{\geq t}$ from \cite[Proposition 10.5.2(7)]{HunekeSwansonIntegralClosure} as the Rees valuations of $\fra S$ restrict to the Rees valuations of $\fra T$.  Likewise $(JS)^{\KH} \cap T = (J T)^{\KH}$ from \cite[Proposition 3.7]{EpsteinMcDonaldRGSchwede}.  
    Therefore \autoref{prop.ConnectionWithKHForParameterIdeals.KHOnS} $\Leftrightarrow$ \autoref{prop.ConnectionWithKHForParameterIdeals.KHOnSSuffBig}.

    Next suppose \autoref{prop.ConnectionWithKHForParameterIdeals.KHOnS} but that \autoref{prop.ConnectionWithKHForParameterIdeals.BiggerThanCJKH} is false, and so $t < c^J_{\KH}(\fra)$.  Choose a finite extension $S$ with $h^{t/m} \in S$ as in \autoref{prop.ConnectionWithKHForParameterIdeals.KHOnSGeneralElement} and so $h^{t/m} \notin (J S)^{\KH}$ by our assumption on $t$.  But $h \in (\fra S)_{\geq m}$ by construction and so $h^{t/m} \in (\fra S)_{\geq t}$.  This is a contradiction and hence \autoref{prop.ConnectionWithKHForParameterIdeals.KHOnS} $\Leftrightarrow$ \autoref{prop.ConnectionWithKHForParameterIdeals.BiggerThanCJKH}.
    
    For the final statements, assume that $J$ is a parameter ideal.  We see that $(J S)^{\KH} = (JS)^{\mathrm{Hir}}$ by \cite[Corollary 6.19]{EpsteinMcDonaldRGSchwede}, and the fact that both Koszul-Hironaka closure and Hironaka pre-closure commute with localization.  This  proves the equivalence of \autoref{prop.ConnectionWithKHForParameterIdeals.KHOnS} with \autoref{prop.ConnectionWithKHForParameterIdeals.HirOnS} and also of \autoref{prop.ConnectionWithKHForParameterIdeals.KHOnSGeneralElement} with  \autoref{prop.ConnectionWithKHForParameterIdeals.GeneralElementHirOnS}.  It also  justifies the ``(equivalently some)'' in \autoref{prop.ConnectionWithKHForParameterIdeals.GeneralElementHirOnS}.
\end{proof}

\begin{proposition}
    \label{prop.KHClosureAgnostic}
    Suppose $J = (x_1, \dots, x_n) \subseteq R$ where $R$ is a Noetherian excellent domain of equal characteristic zero with a dualizing complex, 
    and $\fra 
    \subseteq \sqrt{J}$ is a nonzero ideal.  
    Then 
    \[
        c_{\KH}^{J}(\fra) = c_{\KH}^{J^{\KH}}(\fra).
    \]
\end{proposition}
\begin{proof}
    The inequality $\geq$ follows from \autoref{prop.InitialProperties} \autoref{prop.InitialProperties.BiggerJ}.  
    
    Set $d = c_{\KH}^{J}(\fra) \geq c_{\KH}^{J^{\KH}}(\fra) = c$ and note that both $c$ and $d$ are rational numbers by \autoref{lem.RationalThresholds}.

     For each finite domain extension $R \subseteq S$, we have that $(\fra S)_{\geq c} \subseteq (J^{\KH}S)^{\KH}$ by \autoref{prop.ConnectionWithKHForParameterIdeals}.  But we know that $(J^{\KH}S)^{\KH} \subseteq ((JS)^{\KH})^{\KH} = (JS)^{\KH}$ as Koszul-Hironaka closure is persistent and idempotent, see \autoref{lem.KHHirBasicProperties}.  Hence $c \geq c^J_{\KH}(\fra) = d$ and the result is proven.
\end{proof}

\begin{proposition}
    \label{prop.TwoDefinitionsAgreeForParameterIdeals}
    Suppose $R$ is a Noetherian excellent domain of equal characteristic zero with a dualizing complex, $J = (x_1, \dots, x_n) \subseteq R$ is a parameter ideal,  
    and $\fra \subseteq \sqrt{J}$ is a nonzero ideal.  
    Then 
    \[
         c_{\Hir}^J(\fra) = c_{\KH}^{J}(\fra).
    \]
\end{proposition}
\begin{proof}
    The inequality $\leq$ is \autoref{prop.InitialProperties} \autoref{prop.InitialProperties.ComparisonOfTwoThresholds}.  Set $c = c_{\Hir}^J(\fra)$ and note, by \autoref{prop.BehaviorUnderFinite}, that $c \geq c_{\Hir}^{JS}(\fra S)$ where $S \supseteq R$ is any finite domain extension of $R$.  Fix $Y \to \Spec S$ a log resolution of $(S, \fra S)$ with $\fra \cO_Y = \cO_Y(-D_Y)$.  We know that $S \to S/JS \otimes^{\myL}_S \myR\Gamma(\cO_Y(\lfloor c D_Y \rfloor))$ is zero in the derived category where here we use $c \geq c_{\Hir}^{JS}(\fra S)$.  For any $g \in (\fra S)_{\geq c}$, we see that $\Div_Y(g) \geq c D_Y$ and hence we have a factorization 
    \[
        S \to \myR\Gamma(\cO_Y(\lfloor c D_Y \rfloor)) \to \myR\Gamma(\cO_Y(\Div_Y(g))) 
    \]
    and note that we can identify that composition with $S \xrightarrow{1 \mapsto g} \myR\Gamma(\cO_Y)$.  Thus $S \xrightarrow{1 \mapsto g} S \to S/JS \otimes^{\myL}_S \myR\Gamma(\cO_Y)$ is also zero.  
    Therefore the kernel of $S \to H_0 (S/JS \otimes^{\myL}_S \myR\Gamma(\cO_Y))$ contains $g$, or, in other words, $g \in (JS)^{\Hir}$.  However, as $JS$ is a parameter ideal, we see that $g \in (JS)^{\KH}$ and thus we conclude that $(\fra S)_{\geq c} \subseteq (JS)^{\KH}$.  This implies that $c \geq c_{\KH}^J(\fra)$ by \autoref{prop.ConnectionWithKHForParameterIdeals} and so completes the proof.
\end{proof}

\section{Parameter ideals}

In this section we recover the analog of \cite[Theorem 3.3]{HunekeMustataTakagiWatanabeFThresholdsTightClosureIntClosureMultBounds}.  Although we could also obtain these results via reduction to characteristic $p \gg 0$ (see \autoref{sec.ComparisonWithFThresholds}), we give a purely characteristic zero proof.




\begin{lemma} \label{lem.ParameterDiagonal}
    Suppose $R$ is an excellent Noetherian domain of equal characteristic zero with a dualizing complex.  Suppose $x_1, \dots, x_n$ generate a nonzero parameter ideal $J = (x_1, \dots, x_n) \subsetneq R$ as in \autoref{def.ParameterIdealGeneral}.  Then $c_{\KH}^J(J) = n$.
\end{lemma}

\begin{proof}
    Without loss of generality, and using \autoref{lem.Localization} we may assume that $(R, \fram)$ is local with $x_1, \dots, x_n \in \fram$.  
    In view of \autoref{lem.UpperBound}, it suffices to show that $c_{\KH}^J(J) \geq n$.  
    Suppose for a contradiction that  $c_{\KH}^J(J) < n$.

    Choose successively $h_1,\ldots, h_n\in J$ general $\bQ$-linear combinations of the $x_i$ as in \autoref{prop.GeneralElementComparisonLocal} and set $h = \prod h_i$.  Because the $h_i$ are general and there are $n$ of them, we may also assume that $J = (h_1, \dots, h_n)$.  Let $M$ be a sufficiently large integer such that $c_{\KH}^J(J) < \frac{Mn-1}{M}$ and so $c^J_{\KH}(h) < {Mn - 1 \over Mn}$.
    Consider 
    \[ 
        S = R[h_1^{1/(Mn)}, \dots, h_n^{1/(Mn)}] \subseteq R^+
    \]      
    a finite extension of $R$.  As $R$ is local, so is $S$.  
    Let $Y \to \Spec S$ be a log resolution of $(S, JS)$ and hence a regular alteration over $\Spec R$.  Set $g=\prod h_i^{1/(Mn)} = h^{1/(Mn)}$. By \autoref{prop.ConnectionWithKHForParameterIdeals} \autoref{prop.ConnectionWithKHForParameterIdeals.KHOnSGeneralElement}, \autoref{lem.ZeroVSCohomologicalZero}, the definition of Koszul-Hironaka closure, and since $g^{Mn-1} = h^{Mn-1 \over Mn}$, 
    \[
    S \xrightarrow{1 \mapsto g^{Mn-1}} \Kos({h_1,\ldots, h_n};S) \otimes^{\myL}_S \myR \Gamma(Y, \cO_Y)
    \]
    is zero.  As a consequence, 
    \[ 
        \begin{array}{rcl}
            1 & \in & ((h_1^{Mn/Mn},\ldots, h_n^{Mn/Mn})S)^{\KH}: \prod h_i^{(Mn-1)/(Mn)} 
        \end{array}
    \]
    But now, thanks to \cite[Theorem 4.1]{EpsteinMcDonaldRGSchwede}, $1\in((h_1^{1/(Mn)},\ldots, h_n^{1/(Mn)})S)^{\KH}$.  However, we also have that $\big((h_1^{1/(Mn)},\ldots, h_n^{1/(Mn)})S\big)^{\KH}\subsetneq S$ as Koszul-Hironaka closure is faithful \cite[Proposition 3.3]{EpsteinMcDonaldRGSchwede}. This is a contradiction and so it follows that $c_{\KH}^J(J) \geq n$.
\end{proof}

For the next result, we need some notation.  
\begin{definition}
    \label{def.BEHDefinition}
    Suppose that $R$ is an excellent Noetherian domain of equal characteristic zero  with a dualizing complex and $J = (x_1, \dots, x_n) = (\underline{x}) \neq 0$.  Fix $Z \to \Spec R$ a log resolution of $(R, J)$.  For any integer $\lambda > 0$ we write 
    \[
        \BEH(\underline{x}, \lambda) := \ker \big( R \to H_0(\BE^{\lambda}(\underline{x}) \otimes^{\myL} \myR\Gamma(\cO_Z)) \big)
    \]
    where $\BE^{\lambda}(\underline{x})$ is the Buchsbaum-Eisenbud $L$-complex (or equivalently, an appropriate Eagon-Northcott complex)  which is a free resolution of $R/(x_1, \dots, x_n)^\lambda$ if the $x_i$ form a regular sequence.  
\end{definition}

The object $\BEH(\underline{x}, \lambda)$ appears implicitly in \cite{MaMcDonaldRGSchwede.BrianconSkoda} where the main theorem shows that 
\[
    \overline{J^{\lambda+n-1}} \subseteq \BEH(\underline{x}, \lambda).
\]
Furthermore, $\BEH(\underline{x}, \lambda)$ agrees with the Koszul-Hironaka closure of $J$ if $\lambda = 1$ and in general it is contained in $(J^{\lambda})^{\Hir}$ since we have a natural factorization $R \to \BE^{\lambda}(\underline{x}) \to R/J^{\lambda}$.  Hence we have that 
\begin{equation}
    \label{eq.BEHInIntegralClosure}
    \BEH(\underline{x}, \lambda) \subseteq (J^{\lambda})^{\Hir} \subseteq \overline{J^{\lambda}}
\end{equation}
by \cite[Remark 6.6]{EpsteinMcDonaldRGSchwede}.
As we only need this object as an interim step in our proof, we will not develop its general theory here.  However, we remark that it is reasonable to hope that $\BEH(\underline{x}, \lambda)$ only depends on $J$.  When $\lambda = 1$ this holds by \cite[Proposition 3.3]{EpsteinMcDonaldRGSchwede}.


\begin{lemma}\label{lem.intersectionBEH}
    Suppose $R$ is an excellent Noetherian domain of equal characteristic zero with a dualizing complex and $x_1, \dots, x_n$ generate a nonzero parameter ideal $J \subsetneq R$ in the sense of \autoref{def.ParameterIdealGeneral}.  Then for every integer $\lambda > 0$, 
    \[
        \BEH(\underline{x}, \lambda)=\bigcap_{c_1,\ldots, c_n} (x_1^{c_1},\ldots, x_n^{c_n})^{\KH}
    \]
    where $(c_1, \ldots, c_n)$ runs over all $n$-tuples of positive integers such that $c_1+\cdots+c_n= \lambda+n-1$. 
\end{lemma}
\begin{proof}
    Without loss of generality, we may assume that $R$ is local, as the formation of both sides commutes with localization, and hence $x_1, \dots, x_n$ form part of a system of parameters.  
    Consider the ring $S=\bQ[T_1,\ldots, T_n]_{(T_1,\ldots, T_n)}$ with the map of rings $S\rightarrow R$ such that $T_i\mapsto x_i$ and $r\in \bQ\mapsto r$. Let $\frb=(T_1,\ldots, T_n)$ and set $\Lambda=\{\mathbf{a}=(a_1,\ldots, a_n)\in \bZ_{>0}^n\,|\, a_1+\cdots+a_n=\lambda+n-1\}$.
    A result of Hochster (\cite[Section 3]{LipmanTeissierPseudoRational}, \cite[Lemma 3.5]{RodriguezSchwedeBrianconSkodaViaWeakFunctoriality}) says that 
    $$(T_1,\ldots, T_n)^\lambda=\bigcap_{\mathbf{a}\in \Lambda} (T_1^{a_1},\ldots, T_n^{a_n}),$$
    and so we obtain an exact sequence
    \begin{equation}\label{eq.sec}
        0\rightarrow S/\frb^\lambda\rightarrow \bigoplus_{\mathbf{a}\in \Lambda} S/(T_1^{a_1},\ldots, T_n^{a_n})\rightarrow C\rightarrow 0
    \end{equation}
    where $C$ is a finite-length module over $S$.

    Let $Z$ also denote the (localized) log resolution from \autoref{def.BEHDefinition}.  Note that 
    \begin{align*}
        S/\frb\otimes_S^{\mathbf{L}}\mathbf{R}\Gamma( \mathcal{O}_Z)&\simeq \Kos(T_1,\ldots, T_n; S)\otimes_S^{\mathbf{L}}\mathbf{R}\Gamma( \mathcal{O}_Z)\\
        &\simeq (\Kos(T_1,\ldots, T_n; S)\otimes_S^{\mathbf{L}}R)\otimes_R^{\mathbf{L}}\mathbf{R}\Gamma( \mathcal{O}_Z)\\
        &\simeq \Kos(x_1,\ldots, x_n; R)\otimes_R^{\mathbf{L}}\mathbf{R}\Gamma( \mathcal{O}_Z).
    \end{align*}
    Given that $\mathbf{R}\Gamma( \mathcal{O}_Z)$ is a maximal Cohen-Macaulay $R$-complex, we have that
    \[
        H_i(S/\frb\otimes_S^{\mathbf{L}}\mathbf{R}\Gamma( \mathcal{O}_Z)) = 
        H_i(\Kos(\underline{x}; R) \otimes_R^{\myL}\myR\Gamma(\cO_Z))=0
    \]
    for all $i>0$.
    Additionally, since $C$ has a finite filtration whose successive quotients are $S/\frb$, it follows from a standard chase that 
    $$H_i(C\otimes_S^{\mathbf{L}}\mathbf{R}\Gamma( \mathcal{O}_Z))=0$$
    for all $i>0$.

    Therefore, applying the functor $-\otimes_S^{\mathbf{L}}\mathbf{R}\Gamma( \mathcal{O}_Z)$ to \autoref{eq.sec}, we see that the map 
    $$H_0(S/\frb^\lambda \otimes_S^{\mathbf{L}}\mathbf{R}\Gamma( \mathcal{O}_Z))\rightarrow \bigoplus_{\mathbf{a}\in \Lambda}H_0(S/(T_1^{a_1},\ldots, T_n^{a_n}) \otimes_S^{\mathbf{L}}\mathbf{R}\Gamma( \mathcal{O}_Z))$$
    is injective.

    As $\BE^\lambda(T_1,\ldots, T_n)$ is a free resolution of $S/\frb^\lambda$,
    \begin{align*}
        S/\frb^\lambda\otimes_S^{\mathbf{L}}\mathbf{R}\Gamma( \mathcal{O}_Z)&\simeq \BE^\lambda(T_1,\ldots, T_n)\otimes_S^{\mathbf{L}}\mathbf{R}\Gamma( \mathcal{O}_Z)\\
        &\simeq \BE^\lambda(x_1,\ldots, x_n)\otimes_R^{\mathbf{L}}\mathbf{R}\Gamma( \mathcal{O}_Z)
    \end{align*}    
    and, similarly,
    \begin{align*}
        S/(T_1^{a_1},\ldots, T_n^{a_n})\otimes_S^{\mathbf{L}}\mathbf{R}\Gamma( \mathcal{O}_Z)&\simeq \Kos(x_1^{a_1},\ldots, x_n^{a_n}; R)\otimes_R^{\mathbf{L}}\mathbf{R}\Gamma( \mathcal{O}_Z).
    \end{align*}
    Hence, the map
    $$H_0(\BE^\lambda(x_1,\ldots, x_n)\otimes_R^{\mathbf{L}}\mathbf{R}\Gamma( \mathcal{O}_Z))\rightarrow \bigoplus_{\mathbf{a}\in \Lambda}H_0(\Kos(x_1^{a_1},\ldots, x_n^{a_n}; R)\otimes_R^{\mathbf{L}}\mathbf{R}\Gamma( \mathcal{O}_Z))$$
    is injective and so we see that $\BEH(\underline{x}, \lambda)$ agrees with 
    \begin{align*}
        & \ker\big(R\rightarrow H_0(\BE^\lambda(x_1,\ldots, x_n)\otimes_R^{\mathbf{L}}\mathbf{R}\Gamma( \mathcal{O}_Z))\big)\\
        &=\bigcap_{\mathbf{a}\in \Lambda} \ker\big(R\rightarrow H_0(\Kos(x_1^{a_1},\ldots, x_n^{a_n}; R)\otimes_R^{\mathbf{L}}\mathbf{R}\Gamma( \mathcal{O}_Z))\big)\\
        &= \bigcap_{\mathbf{a}\in \Lambda} (x_1^{a_1},\ldots, x_n^{a_n})^{\KH}.
    \end{align*}
    This completes the proof.
\end{proof}

We now come to our main theorem of the section.

\begin{theorem}
    \label{thm.IntegralClosureMainResult}
    Suppose $R$ is an excellent Noetherian domain of equal characteristic zero with a dualizing complex and $x_1, \dots, x_n$ generate a nonzero parameter ideal $J \subsetneq R$ as in \autoref{def.ParameterIdealGeneral}.  Given an ideal $I$ with  $\sqrt{J} \supseteq I\supseteq J$, we have that $c_{\KH}^J(I) = n$ if and only if $\overline{I}=\overline{J}$.
\end{theorem}
\begin{proof}
    As above, without loss of generality, we may assume that $(R, \fram)$ is local with $J \subseteq \fram$.  
    For the $(\Leftarrow)$ direction, if $\overline{I}=\overline{J}$, then $c_{\KH}^J(I)=c_{\KH}^J(\overline{I})=c_{\KH}^J(\overline{J})=c_{\KH}^J(J)=n$ by \autoref{lem.ParameterDiagonal}.

    Now assume that $c_{{\KH}}^J(I) = n$. By \autoref{prop.ConnectionWithKHForParameterIdeals}, $\overline{I^nS}\subseteq (JS)^{\KH}$ for all finite domain extensions $R\subseteq S$. It follows that $\overline{J^{n-1}S}\cdot\overline{IS}\subseteq\overline{I^{n-1}S}\cdot\overline{IS}\subseteq (JS)^{\KH}$ and so 
    $$\overline{IS} \subseteq ((JS)^{\KH}:_S \overline{J^{n-1}S}).$$

    Fix $t\in \bZ_{>n}$ and let $T = R[x_1^{1/t},\ldots, x_n^{1/t}] \subseteq R^+$.  As $R$ is local, so is $T$.  Furthermore, the maximal ideal $\frn$ of $T$ contains $(x_1^{1/t},\ldots, x_n^{1/t})$.  Now, $\overline{IT} \subseteq ((JT)^{\KH}:_T \overline{J^{n-1}T})$. Since $x_1^{a_1/t} \cdots x_n^{a_n/t}\in \overline{J^{n-1}T}$ for all $a_1,\ldots, a_n\in \bZ_{\geq 0}$ such that $a_1+\cdots+a_n= t(n-1)$, it follows that
    $$\overline{IT}\subseteq \bigcap_{a_1,\ldots, a_n}((JT)^{\KH}:_T x_1^{a_1/t} \cdots  x_n^{a_n/t})$$
    where $(a_1,\ldots, a_n)$ runs over all $n$-tuples of nonnegative integers such that $a_1+\cdots+a_n= t(n-1)$.
    By \cite[Theorem 4.1]{EpsteinMcDonaldRGSchwede},
    $$\overline{IT}\subseteq \bigcap_{a_1,\ldots, a_n} ((x_1^{\frac{t-a_1}{t}},\ldots, x_n^{\frac{t-a_n}{t}})T)^{\KH}$$
    where $(a_1, \ldots, a_n)$ runs over all $n$-tuples of nonnegative integers such that $a_1+\cdots+a_n= t(n-1)$ and $a_i<t$ for all i. As a consequence, 
    $$\overline{IT}\subseteq \bigcap_{c_1,\ldots, c_n} ((x_1^{\frac{c_1}{t}},\ldots, x_n^{\frac{c_n}{t}})T)^{\KH}$$
    where $(c_1, \ldots, c_n)$ runs over all $n$-tuples of positive integers such that $c_1+\cdots+c_n= t$. 
    As $x_1^{1/t}, \dots, x_n^{1/t} \in T$ form a partial system of parameters, \autoref{lem.intersectionBEH} implies that 
    $\overline{IT}\subseteq \BEH(\underline{x}^{1/t}, t - n + 1) \subseteq  \overline{(x_1^{1/t}, \dots, x_n^{1/t})^{t - n + 1}T}$.   
    By intersecting with $R$, and computing fractional integral closures valuatively, we see that   
    \[
        \overline{I} \subseteq J_{\geq {t - n + 1 \over t}}.
    \]
    Sending $t \to \infty$, we see that 
    \[  
        \overline{I} \subseteq \bigcap_{1 > \epsilon > 0}  J_{\geq 1 - \epsilon}.
    \]  
    But the right side is $\overline{J}$ since there are only finitely many Rees valuations computing these rational-power integral closures.
\end{proof}

\section{Jumping numbers}

In this section we compare derived thresholds to the jumping numbers of multiplier ideals in regular rings.  The main technical result, \autoref{ThresholdsAndTest} below, is analogous to \cite[Proposition 2.7]{MustataTakagiWatanabeFThresholdsAndBernsteinSato} or \cite[Proposition 2.29]{BlickleMustataSmithDiscretenessAndRationalityOfFThresholds} in characteristic $p > 0$ and to \cite[Proposition 4.0.5]{RodriguezBCMThresholdsHypersurfaces} or \cite[Proposition 5.4]{RodriguezSchwede.BCMThresholdsNonPrincipal} in mixed characteristic.

We begin with a lemma that should be viewed as an analog of \cite[Corollary 4.3.15]{DattaEpsteinTuckerMittagLefflerAndIntersectionFlatness} and is presumably well known to experts.  A key observation in the proof of this lemma, that there is a natural map $R/J \otimes_R^{\myL} R/J \to R/J \otimes_{R/J}^{\myL} R/J \cong R/J$, was suggested by ChatGPT 5.4 Pro.

\begin{lemma}
    \label{lem.WeirdMapInDerivedCatIsZero}
    Suppose $R$ is a Noetherian ring and $J \subseteq R$ is an ideal such that $J$ is isomorphic to a perfect complex (for instance, if $R$ is regular).  Then the canonical map 
    \[
        \myR\Hom_R(R/J, J) \to \myR\Hom_R(R/J, R)
    \]
    is zero in the derived category.
\end{lemma}
\begin{proof}
    We may view $R/J$ as a perfect complex and hence 
    \[
        \myR\Hom_R(R/J, J) \cong \myR\Hom_R(R/J, R) \otimes^{\myL} J, 
    \]
    see \cite[\href{https://stacks.math.columbia.edu/tag/07VI}{Tag 07VI}]{stacks-project}.
    Hence, if $D = \myR\Hom_R(R/J, R)$, then we must show that $D \otimes_R^{\myL} J \to D \otimes_R^{\myL} R$ is zero.  It thus suffices to show that $D \otimes_R^{\myL} R \to D \otimes_R^{\myL} R/J$ has a left inverse.  But as $D$ can be viewed as a complex of $R/J$-modules, $D \otimes_R^{\myL} R/J \cong D \otimes_{R/J}^{\myL} R/J \otimes^{\myL}_R R/J$ and we have a canonical map $R/J \otimes^{\myL}_R R/J \to R/J$ and thus we have constructed a composition
    \[
        D \otimes_R^{\myL} R \to D \otimes_R^{\myL} R/J \to D \otimes_R^{\myL} R
    \]
    that is the identity as desired.
\end{proof}

We now prove the precise statement.

\begin{proposition}[\textit{cf.} \cite{MustataTakagiWatanabeFThresholdsAndBernsteinSato, BlickleMustataSmithDiscretenessAndRationalityOfFThresholds}]\label{ThresholdsAndTest}
	Suppose $R$ is a finite dimensional Gorenstein excellent domain of equal characteristic zero.
    Let $\fra, J \subseteq R$ be nonzero ideals with $\fra \subseteq \sqrt{J}$.
	\begin{enumerate}
		\item We have that 
        \[ 
            \mJ(R,\fra^{c^J_{\Hir}(\fra)})\subseteq J.
        \]         
        \label{ThresholdsAndTesta}
		\item If $R$ is regular, then for {any real $\alpha > 0$}, 
        \[   
            c_{\Hir}^{\mJ(R, \fra^\alpha)}\left(\fra\right)\leq \alpha.
        \]
        \label{ThresholdsAndTestb}
	\end{enumerate}
\end{proposition}
\begin{proof}
    As $R$ is Gorenstein and finite dimensional, $R$ is a dualizing complex for itself and so we may assume that $K_R = 0$.  
    We first prove \autoref{ThresholdsAndTesta}.  Recall that $\mJ(R, \fra^{c^J_{\Hir}(\fra)})$ is equal to $\Gamma(\cO_Z(\lceil K_Z - c^J_{\Hir}(\fra) D \rceil ))$ where $Z \to \Spec R$ is a log resolution and $\fra \cO_Z = \cO_Z(-D)$.  

   We know that $R \to R/J \otimes^{\myL} \myR\Gamma(\cO_Z(\lfloor c^J_{\Hir}(\fra) D \rfloor))$ is zero.  Fix an element $r \in \mJ(R, \fra^{c^J_{\Hir}(\fra)})$.  Let $\phi : \myR\Gamma(\cO_Z(\lfloor c^J_{\Hir}(\fra) D \rfloor)) \to R$ be the corresponding map as in \autoref{lem.DualOfMultiplierIdeal}.  We have that 
   \[
    R \to R/J \otimes^{\myL} \myR\Gamma(\cO_Z(\lfloor c^J_{\Hir}(\fra) D \rfloor)) \xrightarrow{R/J \otimes^{\myL} \phi} R/J
   \]
   is zero since the first map is zero.  But the composition map is multiplication by $r$, and hence $r \in J$ as desired.


   For part \autoref{ThresholdsAndTestb}, set $J = \mJ(R, \fra^{\alpha})$ and notice that $M := \myR\Gamma(\cO_Z(\lfloor \alpha D \rfloor)) \cong \myR\Hom_R(J, R)$ is the Grothendieck dual of $J$ (\autoref{lem.DualOfMultiplierIdeal}).  We want to show $R \to (R/\mJ(R, \fra^{\alpha})) \otimes^{\myL}_R \myR\Gamma(\cO_Z(\lfloor \alpha D \rfloor)) = R/J \otimes_R^{\myL} M$ is zero.  It suffices to show that $R/J \to R/J \otimes_R^{\myL} M$ is zero.  
   But to show this, it suffices to show that its Grothendieck dual is zero, that is we must show that 
   \[
        \myR\Hom_R(R/J \otimes_R^{\myL} M, R) \to \myR\Hom_R(R/J, R) 
   \]
   is zero.  But $\myR\Hom_R(R/J \otimes_R^{\myL} M, R) \cong \myR\Hom_R(R/J, \myR\Hom_R(M, R)) \cong \myR\Hom_R(R/J, J)$ and so by \autoref{lem.WeirdMapInDerivedCatIsZero} we are done.
\end{proof}

\begin{remark}
    In the proof of \autoref{ThresholdsAndTestb} we use regularity to guarantee that $R/\mJ(R,\fra^{\alpha})$ has finite projective dimension.  If for some reason it had finite projective dimension without that hypothesis, the proof still works.
\end{remark}

We then obtain the following description of the Hironaka threshold purely in terms of the multiplier ideal in a regular ring.  Compare with \cite{MustataTakagiWatanabeFThresholdsAndBernsteinSato,RodriguezSchwede.BCMThresholdsNonPrincipal}.

\begin{corollary}
    \label{cor.HironakaThresholdViaMultiplierIdeal}
    Suppose $R$ is an excellent regular domain of equal characteristic zero with a dualizing complex.
    Let $\fra, J \subseteq R$ be nonzero ideals with $\fra \subseteq \sqrt{J}$.  Then 
    \[
        c_{\Hir}^J(\fra) = \min \{ t \geq 0 \;|\; \mJ(R, \fra^t) \subseteq J \}.
    \]
    Furthermore, if $J$ is proper, then it coincides with $\sup \{ t \geq 0 \;|\; \mJ(R, \fra^t) \not\subseteq J \}$.
\end{corollary}
\begin{proof}
    If $J = R$, then by \autoref{rem.CaseJ=R} there is nothing to show.  Thus we assume $J$ is proper.
    If $t \geq c_{\Hir}^J(\fra)$ then $\mJ(R, \fra^t)\subseteq \mJ(R, \fra^{c^J_{\Hir}(\fra)}) \subseteq J$ by \autoref{ThresholdsAndTest} \autoref{ThresholdsAndTesta}.  Conversely, suppose that $\mJ(R, \fra^t) \subseteq J$.  Then by \autoref{ThresholdsAndTest} \autoref{ThresholdsAndTestb} and \autoref{prop.InitialProperties} \autoref{prop.InitialProperties.BiggerJ}
    \[
        t \geq c_{\Hir}^{\mJ(R, \fra^t)}(\fra) \geq c_{\Hir}^J(\fra)
    \]
    completing the proof.
\end{proof}

We also obtain that the set of thresholds equals the set of jumping numbers in a regular ring.

\begin{corollary}[{\cf \cite{MustataTakagiWatanabeFThresholdsAndBernsteinSato,
    BlickleMustataSmithDiscretenessAndRationalityOfFThresholds}}]
    \label{cor.JumpingEqualThresholds}
    Suppose $R$ is an excellent regular domain of equal characteristic zero with a dualizing complex.  Suppose that $\fra \subseteq R$ is a proper nonzero ideal.  Then the set of jumping numbers for $\fra$ is the same as the set of derived Hironaka thresholds for $\fra$ as the proper ideals $J$, such that $\fra \subseteq \sqrt{J}$, vary.
\end{corollary}
\begin{proof}
    Suppose first that $\alpha > 0$ is a jumping number for $\fra$, so that $\mJ(R, \fra^{\alpha}) \neq \mJ(R, \fra^{\alpha - \epsilon})$ for all $\alpha > \epsilon > 0$.  Then by \autoref{ThresholdsAndTest} \autoref{ThresholdsAndTestb}, $c_{\Hir}^{\mJ(R, \fra^\alpha)}\left(\fra\right)\leq \alpha.$ Set $\alpha'=c_{\Hir}^{\mJ(R, \fra^\alpha)}(\fra)$.
    Then, $\mJ(R,\fra^{\alpha})\subseteq\mJ(R,\fra^{\alpha'})$
    and, by \autoref{ThresholdsAndTest} \autoref{ThresholdsAndTesta}, $\mJ(R,\fra^{\alpha'})\subseteq \mJ(R,\fra^{\alpha})$. Thus, 
    $\mJ(R,\fra^{\alpha'})= \mJ(R,\fra^{\alpha})$ and, since $\alpha$ is a jumping number for $\fra$, it follows that $\alpha=\alpha'=c_{\Hir}^{\mJ(R, \fra^\alpha)}(\fra)$.

    Now let $c=c_{\Hir}^{J}(\fra)$ for some ideal $J$ and suppose that $\mJ(R, \fra^{c'}) = \mJ(R, \fra^{c})$ for some $c'<c$.
    By \autoref{ThresholdsAndTest} \autoref{ThresholdsAndTesta}, $\mJ(R, \fra^{c'}) = \mJ(R, \fra^{c})\subseteq J.$ By \autoref{cor.HironakaThresholdViaMultiplierIdeal}, it follows that $c'\geq c_{\Hir}^{J}(\fra)=c.$
\end{proof}

\subsection{Application to multiplicity}

We also obtain an analog of \cite[Theorem 5.6]{HunekeMustataTakagiWatanabeFThresholdsTightClosureIntClosureMultBounds}, see also \cite[Theorem 6.3]{RodriguezSchwede.BCMThresholdsNonPrincipal}.  The proof is essentially the same as those sources.

\begin{proposition}
    \label{prop.MultiplicityBound}
    Suppose $(R, \fram)$ is a $d$-dimensional excellent regular local ring of equal characteristic zero with a dualizing complex, $\fram = (x_1, \dots, x_d)$, and $J = (x_1^{a_1}, \dots, x_d^{a_d})$ for some positive integers $a_i$.  If $\fra$ is an $\fram$-primary ideal then 
    \[
        e(\fra) \geq \left( {{d \over c_{\Hir}^J(\fra)}} \right)^d e(J)
    \]
    where $e(-)$ denotes Hilbert-Samuel multiplicity.
\end{proposition}
\begin{proof}
    The proof follows closely the argument of \cite[Section 6]{RodriguezSchwede.BCMThresholdsNonPrincipal}, key points of which are due to L.~Ma.  That argument is very similar to \cite[Theorem 5.6]{HunekeMustataTakagiWatanabeFThresholdsTightClosureIntClosureMultBounds}, see also \cite[Lemma 2.3]{deFernexEinMustata.BoundsOnLCThresholdsWithApps} and \cite[Proposition 1.2]{deFernex.LengthMultiplicityAndMultiplier}.  Hence, we only quickly describe why it holds.  The first input is that $c_{\Hir}^J(\fra)$ is computed by determining whether or not $\mJ(R, \fra^t) \subseteq J$ (\autoref{cor.HironakaThresholdViaMultiplierIdeal}).  Then, following the argument of \cite[Theorem 6.3]{RodriguezSchwede.BCMThresholdsNonPrincipal} and  replacing the mixed characteristic test ideal with the multiplier ideal, one only needs certain formal properties of the multiplier ideal to hold.  The only necessary property that does not immediately follow from resolution of singularities (\cite{TemkinDesingularizationOfQuasiExcellentCharZero}) is the restriction theorem for multiplier ideals.  However, the usual argument to prove the restriction theorem still applies in our generality thanks to \cite{Murayama.RelativeVanishingForQSchemes}.  One also uses Howald's formula for multiplier ideals  (\cite{Howald.MultiplierIdealOfMonomial}) applied on the associated graded ring after deformation to the monomial ideal case (see the references above).  We can do this because the associated graded ring is a polynomial ring over a field.  
    
    Alternatively, in the case that $k$ is a field of characteristic zero and $R$ is the localization at a maximal ideal of a finite type $k$-algebra, one can use reduction to characteristic $p \gg 0$ and the results in the next section.
\end{proof}

\section{Comparison with $F$-thresholds}
\label{sec.ComparisonWithFThresholds}

In this section we study our invariants under reduction modulo $p \gg 0$ and compare them with their characteristic $p > 0$ counterparts.

\begin{proposition}
    \label{prop.EasyComparisonReductionModuloP}
    Suppose $R$ is of finite type over a field of characteristic zero $k$ and that $R$ is geometrically integral over $k$, $\fra, J = (x_1, \dots, x_n) \subseteq R$ are proper nonzero ideals, and $\fra \subseteq \sqrt{J}$.  Suppose $A \subseteq k$ is a finitely generated $\bZ$-algebra and $R_A \subseteq R$ is a finite type $A$-algebra and $J_A \subseteq R_A, \fra_A \subseteq R_A$ are ideals with the following properties: the canonical map $R_A \otimes_A k \to R$ is an isomorphism that identifies $J_A \otimes_A k$ with $J$ and identifies $\fra_A \otimes_A k$ with $\fra$.  For each $\lambda \in \mSpec A$, consider the base change $R_{\lambda} = R_A \otimes_A k(\lambda)$, $J_{\lambda}$ the image of $J_A \otimes_A k(\lambda)$ in $R_{\lambda}$, and likewise with $\fra_{\lambda}$.  
    
    Then there is a dense open subset $U \subseteq \mSpec A$ such that $\fra_{\lambda} \subseteq \sqrt{J_{\lambda}}$ and 
    \[
        c^{J_\lambda}_{+}(\fra_{\lambda}) \leq c^J_{\KH}(\fra)
    \]
    for all $\lambda \in U$ (in particular, $c^{J_\lambda}_{*}(\fra_{\lambda}) \leq c^J_{\KH}(\fra)$ by \autoref{lem.ComparisonOfCharpThresholds}).  Finally, if $R$ is Kawamata log terminal, we may pick $U$ as above so that $c^{J_{\lambda}}(\fra_{\lambda}) \leq c^J_{\KH}(\fra)$ for all $\lambda \in U$.
\end{proposition}

The geometrically integral hypothesis guarantees that the mod-$p$-reduction remains a domain.  One could alternately work one component at a time modulo $p \gg 0$ but that complicates the statement.  Of course, if $R$ is normal, one can alternately pass to a finite extension $k \subseteq k'$ so that the base change $R_{k'}$ is a product of geometrically integral domains, and then work on those one at a time.

\begin{proof}
    Set $c := c^J_{\KH}(\fra)$ and recall it is a rational number thanks to \autoref{lem.RationalThresholds}. 
    Set $\pi : Z \to X = \Spec R$ to be a log alteration of $(X, \fra)$ and write $\fra \cO_Z = \cO_Z(-D)$.  Then $R \to \Kos_{\bullet}(\underline{x}) \otimes^{\myL} \myR \Gamma( \cO_Z(\lfloor c D \rfloor))$ is zero by \autoref{lem.WeCanUseLogAlterationInstead} and \autoref{lem.InfsAreMins}.  

    For the moment, enlarge $A$ if necessary so that we can define a log alteration $\pi_A : Z_A \to X_A$ whose base change to $k$ is  $Z \to X$.  Likewise we set $\fra_A \cO_{Z_A} = \cO_{Z_A}(-D_A)$.  We will explain how to reduce to the original $A$ at the end of the proof.  
    After replacing $A$ by $A[g^{-1}]$ (for some nonzero $g \in A$) if necessary we can assume that $R_A$, $J_A = (x_1, \dots, x_n)$, $\fra_A$, $Z_A$, $D_A$, and $\myR \pi_{A*} \cO_{Z_A}(\lfloor cD_A \rfloor)$ base change to the corresponding terms in characteristic zero.  We can furthermore assume that the ring $R_A$, the ideals and their cokernels, as well as the terms of the complex, and the cohomology of the complex, are all flat over $A$.  Note it is harmless to identify the $x_i \in R_A$ generating $J_A$ with their images in $R$.  Furthermore, by considering the localization map 
    \[
        \begin{array}{rl}
            & \Hom_{D(R_A)}(R_A, \Kos_{\bullet}(\underline{x}) \otimes^{\myL}_{R_A} \myR \Gamma(\cO_{Z_A}(\lfloor c D_A \rfloor))) \\
            \to & \Hom_{D(R_A)}(R_A, \Kos_{\bullet}(\underline{x}) \otimes^{\myL}_{R_A} \myR \Gamma(\cO_{Z_A}(\lfloor c D_A \rfloor))) \otimes_A k\\
             \cong & \Hom_{D(R)}(R, \Kos_{\bullet}(\underline{x}) \otimes^{\myL}_{R} \myR \Gamma(\cO_{Z}(\lfloor c D \rfloor)))
        \end{array}
    \]
    and noting that the left term is simply $\myH^0(\Kos_{\bullet}(\underline{x}) \otimes^{\myL} \myR \Gamma(\cO_{Z_A}(\lfloor c D_A \rfloor)))$ which we may also assume is flat over $A$, we see we may assume that $R_A \to \Kos_{\bullet}(\underline{x}) \otimes^{\myL} \myR \Gamma(\cO_{Z_A}(\lfloor c D_A \rfloor))$ is zero as well.

    We then have that $\myR \Gamma(\cO_{Z_A}(\lfloor cD_A \rfloor))\otimes_A^{\myL} A/\lambda \cong \myR \Gamma(\cO_{Z_\lambda}(\lfloor cD_{\lambda} \rfloor))$.  Here $(-)_{\lambda}$ denotes the corresponding object base changed by $\otimes_A A/\lambda$, in other words it is  
    the base change to characteristic $p > 0$.  Hence 
    \[
        R_{\lambda} \to  \Kos_{\bullet}(\underline{x}) \otimes^{\myL} \myR \Gamma( \cO_{Z_{\lambda}}(\lfloor c D_{\lambda} \rfloor))
    \]
    is also zero (again, we identify the $x_i$ with their images).
    
    Now, pick $f \in (\fra_{\lambda} R_{\lambda}^+)_{> c}$.  Pick $R_{\lambda} \to R' \subseteq R_{\lambda}^+$ a finite domain extension with $f \in R'$.  Let $Z' \to X' = \Spec R'$ denote an alteration obtained by normalizing a component of $X' \times_{X_{\lambda}} Z_{\lambda}$ which dominates $X'$.  Set $\fra_{\lambda} \cO_{Z'} = \cO_{Z'}(-D')$ so that $D'$ is the pullback of $D_{\lambda}$. 
    We see that $\Div_{Z'} f \geq \lfloor c D' \rfloor$ 
    and hence we have that the multiplication-by-$f$-map factors as follows   
    \[
        \times f : \myR\Gamma(\cO_{Z'}) \to \myR\Gamma( \cO_{Z'} (\lfloor c D' \rfloor)) \to \myR\Gamma( \cO_{Z'}).
    \]
    Hence we obtain the commutative diagram:
    \[
        \xymatrix{
            R_{\lambda} \ar[d] \ar[r] & R' \ar[r] & \myR\Gamma( \cO_{Z'}) \ar[d] \ar[r]^{\times f} & \myR\Gamma(\cO_{Z'}) \\
            \myR\Gamma( \cO_{Z_{\lambda}}(\lfloor c D_{\lambda}\rfloor)) \ar[rr] & & \myR\Gamma( \cO_{Z'}(\lfloor c D' \rfloor)) \ar[ur]
        } 
    \]
    Hence we see that $R_{\lambda} \xrightarrow{1 \mapsto f} \Kos_{\bullet}(\underline{x}) \otimes^{\myL} \myR\Gamma( \cO_{Z'})$ is zero.  But we have a factorization $R_{\lambda} \to \Gamma(\cO_{Z'}) \to \myR\Gamma( \cO_{Z'}) \to R_{\lambda}^+$ by \cite{BhattDerivedDirectSummand}.  Hence 
    \[
    R_{\lambda} \xrightarrow{1 \mapsto f} \Kos_{\bullet}(\underline{x}) \otimes R_{\lambda}^+
    \]
    is zero.  Taking zeroth cohomology proves that $f \in J R_{\lambda}^+$ and so completes the proof for our potentially larger $A$ (taking $U$ to be the open subset of $\mSpec A$ obtained by inverting the various $g$ as above).
    

    Now we explain why the statement is independent of the choice of $A$.  It suffices to prove the following.  Suppose that $A \subseteq A' \subseteq k$ are finite type $\bZ$-algebras and that we can construct $R_A$, $J_A$ and $\fra_A$ as in the statement and $R_{A'} = R_A \otimes_A A'$, and that $J_{A'}$, $\fra_{A'}$ are the ideals generated by the images of $J_A$ and $\fra_A$.  
    We shall show that the theorem holds for $A$ if and only if it holds for $A'$, which suffices to complete the proof.
     After inverting a nonzero element of $A$, we may assume that $A \to A'$ is flat.     
     
    Let $\phi : \mSpec A' \to \mSpec A$ denote the natural map and suppose $\phi (\lambda') = \lambda$.  We notice that the induced 
    $R_{\lambda} \to R_{\lambda'}$ is finite \'etale as $k(\lambda) \subseteq k(\lambda')$ is an extension of finite fields (since $A$ and $A'$ are finite type $\bZ$-algebras).  
    But we also have that  
    \[
        c_+^{J_{\lambda}}(\fra_{\lambda}) = \inf\{ t \in \bQ_{\geq 0} \;|\; (\fra_{\lambda} R_{\lambda}^+)_{> t} \subseteq J_{\lambda} R_{\lambda}^+ \} = c_+^{J_{\lambda'}}(\fra_{\lambda'}) 
    \]
    since $R_{\lambda} \subseteq R_{\lambda'} \subseteq R_{\lambda}^+$ and in that extension $\fra_{\lambda} R_{\lambda'} = \fra_{\lambda'}$ and likewise $J_{\lambda} R_{\lambda'} = J_{\lambda'}$.  
    
    Thus if $U \subseteq \mSpec A$, and if the result holds for points of $U$, then it also holds for points of $\phi^{-1}(U) \subseteq \mSpec A'$.  On the other hand, if we have the statement for some dense open $U' \subseteq \mSpec A'$, then it also holds for all points in $\phi(U')\subseteq \mSpec A$ which is open and dense as $A \to A'$ is a finite type flat map of domains.

    The final statement about Kawamata log terminal singularities follows since then the characteristic $p > 0$ models are strongly $F$-regular (\cite{HaraRatImpliesFRat,MehtaSrinivasRatImpliesFRat,HaraWatanabeFRegFPure,TakagiInterpretationOfMultiplierIdeals}), and thus weakly $F$-regular, so \autoref{lem.ComparisonOfCharpThresholds} applies.
\end{proof}

\begin{theorem}
    \label{thm.ReductionModuloP}
    We adopt the notations and assumptions of \autoref{prop.EasyComparisonReductionModuloP}.  Additionally assume that $J$ is a parameter ideal (as in \autoref{def.ParameterIdealGeneral}).  Then for every rational $\epsilon > 0$, there is a dense open set $U_{\epsilon} \subseteq \mSpec A$ such that 
    \[
        c^{J_\lambda}_{+}(\fra_{\lambda}) \geq c^J_{\KH}(\fra) - \epsilon
    \]
    for all $\lambda \in U_{\epsilon}$.    Informally speaking, or if $R, J, \fra$ are defined over $\bZ$, this means that 
    \[
        \lim_{p \to \infty}  c^{J_p}_{+}(\fra_{p}) = \lim_{p \to \infty}  c^{J_p}_{*}(\fra_{p}) = c^J_{\KH}(\fra).
    \]
\end{theorem}   
In particular, if $R$ is of weakly  $F$-regular type, we may pick $U_{\epsilon}$ as above so that $c^{J_{\lambda}}(\fra_{\lambda}) \geq c^J_{\KH}(\fra) - \epsilon$ for all $\lambda \in U_{\epsilon}$.
\begin{proof}
    Fix $\epsilon > 0$ and set $c_0 := c^J_{\KH}(\fra)$ recalling it is in $\bQ$.  We may assume that $t = c_0 - \epsilon > 0$ as otherwise the statement is vacuously true.  Pick an integer $m > c_0 = c^J_{\KH}(\fra)$ and $h \in R$ as in \autoref{cor.GeneralElementChoiceNew} so that $m c^J_{\KH}(h) = c^J_{\KH}(\fra)$.  
    
    Pick $S \supseteq R$ a normal finite extension such that $h^{t/m} \in S$ and observe that $h^{t/m} \in (\fra S)_{\geq t}$.  By \autoref{prop.ConnectionWithKHForParameterIdeals}, we see that $h^{t/m} \notin (JS)^{\KH}$.  After reducing this whole setup to characteristic $p \gg 0$ as well (possibly enlarging $A$ and shrinking $U_{\epsilon}$), and using \cite[Theorem 7.12]{EpsteinMcDonaldRGSchwede} we see that $(h^{t/m})_{\lambda} \notin (J_{\lambda}S_{\lambda})^* = (J_{\lambda}S_{\lambda})^+$ for all $\lambda \in U_{\epsilon}$.  
    As $S_{\lambda}$ need not be a domain, we explain exactly what we mean by this. As $S$ is normal, after shrinking $U$ if necessary, we may assume that $S_{\lambda}$ is normal and hence is a finite product of domains $S_{\lambda,i}$ each a finite domain extension of $R_{\lambda}$, and so we define $S_{\lambda}^+ = \prod S_{\lambda,i}^+$.  
    Fix one such $\lambda$ and $S_{\lambda,i}$ with $(h^{t/m})_{\lambda}S_{\lambda,i} \not\subseteq (J_{\lambda}S_{\lambda,i})^+$.  Since $(h^{t/m})_{\lambda}S_{\lambda,i} \subseteq (\fra_{\lambda} S_{\lambda,i})_{\geq t}$, we see that $(\fra_{\lambda} S_{\lambda,i})_{\geq t} \nsubseteq (J_{\lambda}S_{\lambda,i})^{+}$.  It follows that $(\fra_{\lambda} S_{\lambda,i})_{>t - \delta} \not\subseteq (J_{\lambda} S_{\lambda,i})^+$ for all sufficiently small rational $\delta > 0$. Thus  
    by \autoref{def.CharPThresholds}, we obtain
        \[  
            t \leq c_{+}^{J_{\lambda}}(\fra_{\lambda}).
        \]  
    This completes the proof of the main statement.  For the final statement, note that tight closure equals $+$-closure for parameter ideals in locally excellent domains \cite[Theorem 5.1]{SmithTightClosureParameter} (and in particular, for domains of finite type over a field).
\end{proof}

\subsection{Regular ambient rings}

We point out that, in the case of a regular ring, similar results hold for the Hironaka threshold even without the parameter ideal hypothesis.  The proof is a simple consequence of \cite[Theorem 3.4]{HaraYoshidaGeneralizationOfTightClosure} and \autoref{cor.HironakaThresholdViaMultiplierIdeal}.

\begin{theorem}
    Suppose that $R$ is a regular domain of finite type and geometrically integral over a field $k$ of characteristic zero and $\fra, J \subseteq R$ are nonzero proper ideals such that $\fra \subseteq \sqrt{J}$.  Suppose $A \subseteq k$ is a finitely generated $\bZ$-algebra with $R_A, J_A, \fra_A$ as in \autoref{prop.EasyComparisonReductionModuloP}.  There is a dense open set $U \subseteq \mSpec A$ such that 
    \[
        c^{J_{\lambda}}(\fra_{\lambda}) \leq c_{\Hir}^J(\fra)
    \]
    for all $\lambda \in U$ (here the left side is the ordinary $F$-threshold).  Furthermore, for every sufficiently small rational $\epsilon > 0$, there is a dense open set $U_\epsilon \subseteq \mSpec A$ such that 
    \[
        c^{J_\lambda}(\fra_{\lambda}) > c^J_{\Hir}(\fra) - \epsilon
    \]
    for all $\lambda \in U_{\epsilon}$.  In other words, roughly speaking,
    \[
        \lim_{p \to \infty} c^{J_p}(\fra_{p}) = c^J_{\Hir}(\fra).
    \]
\end{theorem}
\begin{proof}
    Fix $c_0 = c_{\Hir}^J(\fra) \in \bQ$.  Then $\mJ(R, \fra^{c_0}) \subseteq J$ by \autoref{ThresholdsAndTest} \autoref{ThresholdsAndTesta}.  Thanks to \cite[Theorem 3.4]{HaraYoshidaGeneralizationOfTightClosure}, we see that $\mJ(R, \fra^{c_0})_{\lambda} = \tau(R_{\lambda}, \fra_{\lambda}^{c_0})$ for all $\lambda$ in some dense open $U \subseteq \mSpec A$.  But then $c^{J_{\lambda}}(\fra_{\lambda}) \leq c_0$ by \cite[Proposition 2.7]{MustataTakagiWatanabeFThresholdsAndBernsteinSato} (reformulated exactly as in the proof of \autoref{cor.HironakaThresholdViaMultiplierIdeal}).  

    Now, suppose that $c_0 > \epsilon > 0$ and $t = c_0 - \epsilon$.  Then $\mJ(R, \fra^t) \not\subseteq J$ thanks to \autoref{cor.HironakaThresholdViaMultiplierIdeal}.  Fix $U_{\epsilon} \subseteq \mSpec A$ such that $\mJ(R, \fra^t)_{\lambda} = \tau(R_{\lambda}, \fra_{\lambda}^t)$  and such that $\mJ(R, \fra^t)_{\lambda} \not\subseteq J_{\lambda}$ for all $\lambda \in U_{\epsilon}$.  But then, again by \cite[Proposition 2.7]{MustataTakagiWatanabeFThresholdsAndBernsteinSato}, we see that $t < c^{J_{\lambda}}(\fra_{\lambda})$ which is what we wanted to show.
\end{proof}

\bibliographystyle{skalpha}
\bibliography{main}

\end{document}